\documentclass[EJP]{ejpecp} 
\DeclareMathOperator\ess{ess}
\DeclareMathOperator\supp{supp}
\DeclareMathOperator\conv{conv} %

\SHORTTITLE{The Slepian zeros and Brownian bridge embedded in Brownian motion} 

\TITLE{The Slepian zero set, and Brownian bridge embedded in Brownian motion by a spacetime shift} 

\AUTHORS{
 Jim~Pitman\footnote{Statistics department, University of California, Berkeley.
  \EMAIL{pitman@stat.berkeley.edu}}
  \and
 Wenpin~Tang\footnote{Statistics department, University of California, Berkeley.
  \EMAIL{wenpintang@stat.berkeley.edu}}}


\AMSSUBJ{60G17; 60J55; 60J65} 

\SUBMITTED{November 6, 2014}
\ACCEPTED{June 9, 2015}

\VOLUME{20}
\YEAR{2015}
\PAPERNUM{61}
\DOI{v20-3911}

\ABSTRACT{This paper is concerned with various aspects of the Slepian process $(B_{t+1} - B_t,  t \ge 0)$ derived from one-dimensional standard Brownian motion $(B_t, t \ge 0 )$. In particular, we offer an analysis of the local structure of the Slepian zero set $\{t : B_{t+1} = B_t \}$, including a path decomposition of the Slepian process for $0 \le t \le 1$. We also establish the existence of a random time $T \geq 0$ such that the process $(B_{T+u} - B_T, 0 \le u \le 1)$ is standard Brownian bridge.}

\begin{document}



\section{Introduction and main result}
In a recent work \cite{PTpattern}, we were interested in continuous paths of length $1$ in Brownian motion $(B_t;t \geq 0)$. We proved that Brownian meander $m$ and the three-dimensional Bessel process $R$ can be embedded into Brownian motion by a random translation of origin in spacetime, while it is not the case for either normalized Brownian excursion $e$ or reflected Brownian bridge $|b^0|$. The following question was left:
\begin{question}
\label{q1}
Can we find a random time $T \geq 0$ such that $(B_{T+u}-B_T; 0 \leq u \leq 1)$ has the same distribution as standard Brownian bridge $(b_{u}^0;0 \leq u \leq 1)$?
\end{question}

As a natural candidate, the bridge-like process as below was considered:
\begin{equation}
\label{bridgelike}
(B_{F+u}-B_F; 0 \leq u \leq 1), 
\end{equation}
where
\begin{equation}
\label{bridgeT}
F:=\inf\{t \geq 0; B_{t+1}-B_t=0\}.
\end{equation}

This bridge-like process bears some resemblance to Brownian bridge. At least, it starts and ends at $0$,
and is some part of a Brownian path in between. This leads us to the following question: 
\begin{question}~
\label{opbr}
\begin{enumerate}
\item
\label{op411}
Is the bridge-like process defined as in \eqref{bridgelike} standard Brownian bridge?
\item
\label{op412}
If not, is the distribution of standard Brownian bridge absolutely continuous with respect to that of the bridge-like process?
\end{enumerate}
\end{question}

Let $\mathcal{C}_0[0,1]$ be the set of continuous paths $(w_t; 0 \leq t \leq 1)$ starting at $w_0=0$, and $\mathcal{B}$ be the Borel $\sigma-$field of $\mathcal{C}_{0}[0,1]$. To provide a context for the above questions, we observe that
\begin{equation}
\label{MWT}
F:=\inf\{t \geq 0; X_t \in \mathcal{BR}^0\},
\end{equation}
where 
\begin{equation}
\label{mmm}
X_t:=(B_{t+u}-B_t; 0 \leq u \leq 1) \quad \mbox{for}~ t \geq 0,
\end{equation}
is the moving-window process associated to Brownian motion $(B_t; t \geq 0)$, and 
$$
\mathcal{BR}^0:=\{w \in \mathcal{C}_0[0,1]; w(1)=0\}
$$ is the set of bridges with endpoint $0$. Note that the moving-window process $X$ is a stationary Markov process, with transition kernel $P_t: (\mathcal{C}_0[0,1],\mathcal{B}) \rightarrow (\mathcal{C}_0[0,1],\mathcal{B})$ for $t \geq 0$ given by
$$
P_t(w,d\widetilde{w}) =     \left\{ \begin{array}{ccl}          \mathbb{P}^{{\bf W}}(d\widetilde{w})& \mbox{if}         & t \geq 1, \\ 1(\widetilde{w}= (w_{t+u}-w_t;u \leq 1-t) \otimes \widetilde{w}')\mathbb{P}^{\bf W_t}(d\widetilde{w}')\  & \mbox{if} & t <1, \\                \end{array}\right.
$$
where $\mathbb{P}^{{\bf W}}$ (resp. $\mathbb{P}^{\bf W_t}$) is Wiener measure on $\mathcal{C}_0[0,1]$ (resp. $\mathcal{C}_0[0,t]$), and $\otimes$ is the usual path concatenation.
Note that $\mathbb{P}^{\bf W}$ is invariant with respect to $(P_t; t \geq 0)$. 
Moreover, $X_{t+l}$ and $X_t$ are independent for all $t \geq 0$ and $l \geq 1$.

For a suitably nice continuous-time Markov process $(Z_t; t \geq 0)$, there have been extensive studies on the post-$T$ process $(Z_{T+t}; t \geq 0)$ with some random time $T$ which is
\begin{itemize}
\item
a stopping time, see e.g. Hunt \cite{Hunt} for Brownian motion, Blumenthal \cite{Blumenthal}, and Dynkin and Jushkevich \cite{DJ} for general Markov processes;
\item
an honest time, that is the time of last exit from a predictable set, see e.g. Meyer et al. \cite{MSW}, Pittenger and Shih \cite{PSbis,PS}, Getoor and Sharpe \cite{GetSh,GetShtri,GetShbis}, Maisonneuve \cite{Maison} and Getoor \cite{Getoorsplit};
\item
the time at which $X$ reaches its ultimate minimum, see e.g. Williams \cite{Williams2} and Jacobsen \cite{Jacob} for diffusions, Pitman \cite{Pitmantech} for conditioned Brownian motion and Millar \cite{Millarbis,Millar} for general Markov processes.
\end{itemize}

The problem is related to decomposition/splitting theorems of Markov processes. We refer readers to the survey of Millar \cite{Millarsurvey}, which contains a unified approach to most if not all of the above cases. See also Pitman \cite{PitmanLevy} for a presentation in terms of point processes and further references. Moreover, if $Z$ is a semi-martingale and $T$ is an honest time, the semi-martingale decomposition of the post-$T$ process was investigated in the context of progressive enlargement of filtrations, by Barlow \cite{Barlow}, Yor \cite{Yor}, Jeulin and Yor \cite{JeulinYor} and in the monograph of Jeulin \cite{Jeulin}. The monograph of Mansuy and Yor \cite{MY} offers a survey of this theory.

The study of the bridge-like process is challenging, because the random time $F$ as in \eqref{bridgeT} does not fit into any of the above classes. We even do not know whether this bridge-like process is Markov, or whether it enjoys the semi-martingale property. Note that if the answer to \eqref{op412} of Question \ref{opbr} is positive, then we can apply {\em Rost's filling scheme} \cite{CO,Rost1} as in Pitman and Tang \cite[Section $3.5$]{PTpattern} to sample Brownian bridge from a sequence of i.i.d. bridge-like processes in Brownian motion by iteration of the construction \eqref{bridgelike}. While we are unable to answer either of the above questions about the bridge-like process, we are able to settle Question \ref{q1}.
\begin{theorem}
\label{mainbis}
There exists a random time $T \geq 0$ such that $(B_{T+u}-B_T; 0 \leq u \leq 1)$ has the same distribution as $(b^0_u; 0 \leq u \leq 1)$.
\end{theorem}

In terms of the moving-window process, it is equivalent to find a random time $T \geq 0$ such that $X_T$ has the same distribution as Brownian bridge $b^0$. As mentioned in Pitman and Tang \cite{PTpattern}, this is a variant of the {\em Skorokhod embedding problem} for the $\mathcal{C}_0[0,1]$-valued process $X$ and a random time $T$. Our proof relies on Last and Thorisson's 
construction 
\cite{LT2014}
of the {\em Palm measure} of local times of the moving-window process. The idea of embedding {\em Palm/Revuz measures} arose earlier in the work of Bertoin and Le Jan \cite{BertoinLeJan}, and the connection between Palm measures and Markovian bridges was made by Fitzsimmons et al. \cite{FPY}. The existence of local times stems from the Brownian structure of the zero set of the {\em Slepian process} $S_t:=B_{t+1} - B_t$ for $t \geq 0$, which is introduced in Section \ref{slepiansec}. See e.g. Theorem \ref{main} and Lemma \ref{abszero}.

In the proof of Theorem \ref{mainbis}, we make use of an abstract existence result of Thorisson \cite{Th2}, see e.g. Theorem \ref{thts}. As a consequence, our method gives little clue on the construction of the random time $T$. However, while the paper was under review, we learned from Hermann Thorisson \cite{HTP} an explicit embedding of Brownian bridge into Brownian motion by a spacetime shift. His argument is mainly from Last et al. \cite{LMT}, that is to construct allocation rules balancing stationary diffuse random measures on the real line. In particular, they were able to characterize {\em unbiased shifts} of Brownian motion, those are random times $T \in \mathbb{R}$ such that $(B_{T+u}-B_T; u \geq 0)$ is a two-sided Brownian motion, independent of $B_T$. Though it appears to be easier, Thorisson's constructive proof relies on a deep and powerful theory. As the two proofs of Theorem \ref{mainbis} are of independent interest, we present both ours in Subsections \ref{32} and \ref{33}, and Thorisson's in Subsection \ref{34}.

We conclude this introductory part by reviewing related literature. Question \ref{q1} is closely related to the notion of {\em shift-coupling}, initiated by Aldous and Thorisson \cite{AT}, and Thorisson \cite{Th1}. General results of shift-coupling were further developed by Thorisson \cite{Th1995,Th2, Th1999}, see also the book of Thorisson \cite{Thbook}. In the special cases of a family of i.i.d. Bernoulli random variables indexed by $\mathbb{Z}^d$ or a spatial Poisson process on $\mathbb{R}^d$, Liggett \cite{Liggett}, and Holroyd and Liggett \cite{HL} provided an explicit construction of the random shift and computed the tail of its probability distribution. Two continuous processes $(Z_u; u \geq 0)$ and $(Z'_u; u \geq 0)$ are said to be shift-coupled if there are random times $T,T' \geq 0$ such that $(Z_{T+u}; u \geq 0)$ has the same distribution as $(Z'_{T'+u}; u \geq 0)$. From Theorem \ref{mainbis}, we know that $Z:=X$, the moving-window process can be shift-coupled with some $\mathcal{C}_0[0,1]$-valued process $Z'$ starting at $Z'_0:=b^0$ for random times $T \geq 0$ and $T'=0$. 

More recently, Hammond et al. \cite{HPS} constructed local times on the exceptional times of two dimensional dynamical percolation. Further, they showed that at a typical time with respect to local times, the percolation configuration has the same distribution as {\em Kesten's Incipient Infinite Cluster} \cite{Kestenperco}. They also made use of Palm theory and the  idea was similar in spirit to ours, though the framework is completely different. In a study of {\em forward Brownian motion}, Burdzy and Scheutzow \cite{Burdzy} asked whether a concatenation of independent pieces of Brownian paths truncated at stopping times forms Brownian motion. They showed that if these Brownian pieces are i.i.d. and the expected stopping times are finite, then Brownian motion is achieved by patchwork. The general case where the Brownian pieces are not identically distributed, is left open.\\\\
\textbf{Organization of the paper:} The rest of this paper is organized as follows.
\begin{itemize}
\item 
In Section \ref{rwsec}, we present some analysis of random walks related to Question \ref{opbr}. 
\item
In Section \ref{slepiansec}, after recalling some results for the Slepian process due to Slepian \cite{Slepian} and Shepp \cite{Shepp}, we provide a path decomposition for the Slepian process on $[0,1]$, Theorem \ref{main}.
\item
In Section \ref{2}, we explore the local structure of the Slepian zero set $\{t \in [0,1]; S_t=0\}$, or $\{t \in [0,1]; X_t \in \mathcal{BR}^0\}$. In particular, Theorem \ref{main} is proved in Section \ref{22}.
\item
In Section \ref{3}, after presenting essential background on Palm theory of stationary random measures, we give two proofs of Theorem \ref{mainbis}. The constructive proof in Subsection \ref{34} is due to Hermann Thorisson.
\end{itemize}
\section{Random walk approximation}
\label{rwsec}
In this section, we consider the discrete analog of the bridge-like process. Namely, for an even positive integer $n$, we run a simple symmetric random walk $(RW_k)_{k \in \mathbb{N}}$ until the first level bridge of length $n$ appears. That is, we consider the process
\begin{equation}
\label{disbl}
(RW_{F_n+k}-RW_{F_n})_{0 \leq k \leq n}, \quad \mbox{where}~F_n:=\inf\{k \geq 0; RW_{k+n}=RW_k\}.
\end{equation}
The following invariance principle is proved in Pitman and Tang \cite[Proposition $2.4$]{PTpattern}.
\begin{proposition}
\label{PT2014}
\cite{PTpattern}
The distribution of the process
\label{scaling}
$$\left(\frac{RW_{F_n+nu}-RW_{F_n}}{\sqrt{n}}; 0 \leq u \leq 1 \right)
$$
where the walk is defined by linear interpolation between integer times, converges weakly to the distribution of the bridge-like process as in \eqref{bridgelike}.
\end{proposition}

Further, we may consider Knight's \cite{Knightapprox1,Knightapprox2} embedding of random walks in Brownian motion. Endow the space $\mathcal{C}[0,\infty)$ with the topology of uniform convergence on compact sets. Fix $n \in \mathbb{N}$. Let $\tau^{(n)}_0:=0$ and $\tau_{k+1}^{(n)}:=\inf\{t>\tau_k^{(n)}; |B_t-B_{\tau_k^{(n)}}|=2^{-n} \}$ for $k \in \mathbb{N}$. Note that $\left(RW^{(n)}_k:=2^nB_{\tau^{(n)}_k}\right)_{k \in \mathbb{N}}$ is a simple random walk. In addition, the sequence of linearly interpolated random walks
$$\left(\frac{RW^{(n)}_{2^{2n}t}}{2^n} ; t \geq 0 \right)~\mbox{converges almost surely in}~\mathcal{C}[0,\infty)~\mbox{to}~(B_t; t \geq 0).$$
It is not hard to see that $F(w):=\inf\{t \geq 0; w_{t+1}=w_t\}$ is not continuous at all paths $w \in \mathcal{C}_0[0,\infty)$. Nevertheless, Pitman and Tang \cite[Proposition $2.4$]{PTpattern} proved that $F$ is $\mathbb{P}^{\bf W}$-a.s. continuous, where $\mathbb{P}^{\bf W}$ is  Wiener measure on $\mathcal{C}[0,\infty)$. Thus, the convergence of Proposition \ref{PT2014} is almost sure in the context of Knight's construction of simple random walks.

Now we focus on the discrete bridge defined as in \eqref{disbl}. Note that the support of the first level bridge is all bridge paths since the first $n$ steps starting from $0$ can be any path. For $n=2$, the bridge $(RW_{F_2+k}-RW_{F_2})_{0 \leq k \leq 2}$ obviously has uniform distribution on the two possible paths, one positive and one negative.  However, the first level bridge of length $n$ is not uniform for $n > 2$. Using the Markov chain matrix method, we can compute the exact distribution of this first level bridge for $n=4$ and $6$. By up-down symmetry, we only need to be concerned with those paths whose first step is $+1$.
\begin{center}
\includegraphics[width=0.5 \textwidth]{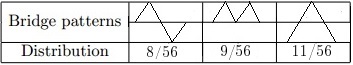}\\
TABLE $1$. The distribution of the first level bridge as in \eqref{disbl} for $n=4$.
\end{center}

\begin{center}
\includegraphics[width=1 \textwidth]{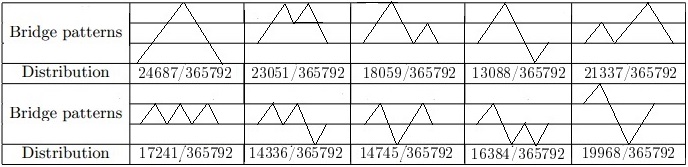}\\
TABLE $2$. The distribution of the first level bridge as in \eqref{disbl} for $n=6$.
\end{center}

The numerical results in Table $1$ and $2$ give us that the first level bridge fails to be uniform, at least, for $n=4$ and $6$. By elementary algebraic computation, it is not hard to check that this is true for all $n >2$. Now it is natural to ask whether the first level bridge could be asymptotically uniform. To this end, we compute the ratio of extremal probabilities of the first level bridge for some small $n$'s.
\begin{center}
\includegraphics[width=0.5\textwidth]{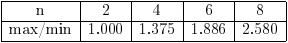}\\
TABLE $3$. The ratio max/min probability of the first level bridge of length $n$.
\end{center}

In Table $3$, the ratios max/min of hitting probabilities suggest that the first level bridge might not be asymptotically uniform. Thus, the answer to \eqref{op411} of Question \ref{opbr} may be negative, i.e. the bridge-like process defined as in \eqref{bridgelike} is not standard Brownian bridge. 

This is further confirmed by the following simulations, which show that as $n$ grows, the empirical distribution of the maximum of the first level bridge does not appear to converge to the {\em Kolmogorov-Smirnov distribution} \cite{Kolm,Smirnov}, that is the distribution of the supremum of Brownian bridge, see e.g. Billingsley \cite[Section $13$]{Bill}.
\begin{figure}[h]
\begin{center}
\includegraphics[width=0.6 \textwidth]{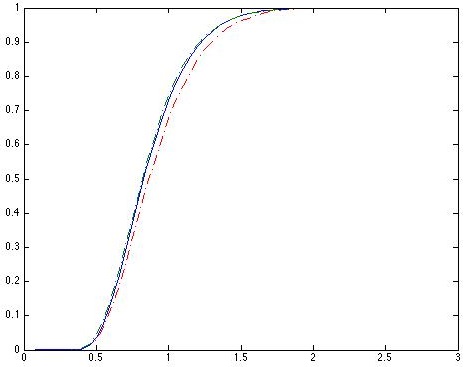}\\
\end{center}
\caption{Solid curve: the Kolmogorov-Smirnov CDF; Dashed curve over the solid curve: the empirical CDF of the maximum of scaled uniform bridge of length $n=10^4$; dashed curve below the solid curve: the empirical CDF of the maximum of the first level bridge of length $n=10^4$.}
\end{figure}
\begin{center}
    \begin{tabular}{| c | c | c | c | c | c | c |}
    \hline
    $n$ & $100$ & $500$ & $1000$ & $2000$ & $5000$ & $10000$ \\ \hline
    CDF($1.3$) & $0.9361$ & $0.9193$ & $0.9129$ & $0.9117$ & $0.9088$ & $0.9080$ \\ \hline
    Difference & $-0.0042$ & $0.0126$ & $0.0190$ & $0.0202$ & $0.0231$ & $0.0239$ \\ \hline
    \end{tabular}\\
    \end{center}
{TABLE $4$. $2^{nd}$ row: the CDFs at $1.3$ of the scaled maximum of the first level bridge of length $n$. $3^{rd}$ row: the differences between the Kolmogorov-Smirnov CDF evaluated at $1.3$ ($\approx 0.9319$) and those of the $2^{nd}$ row.}
\section{The Slepian process: Old and New}
\label{slepiansec}
 Let us turn back to the random time $F$ defined as in \eqref{bridgeT}. We rewrite it as
\begin{equation}
\label{Ts}
F:=\inf\{t \geq 0; S_t=0\},
\end{equation}
where $S_t:=B_{t+1}-B_{t}$ for $t \geq 0$ is a stationary Gaussian process with mean $0$ and covariance $\mathbb{E}[S_{t_1}S_{t_2}]=\max(1-|t_1-t_2|,0)$. Note that $(S_t; t \geq 0)$ is not Markov, since the only nontrivial stationary, Gaussian and Markov process is the Ornstein-Uhlenbeck process, see e.g. Doob \cite[Theorem $1.1$]{Doob}. The process $(S_t; t \geq 0)$ was first studied by Slepian \cite{Slepian}. Later, Shepp \cite{Shepp} gave an explicit formula for
$$I(t|x):=\mathbb{P}(F>t|S_0=x),$$
as a $t-$fold integral when $t$ is an integer and as a $2[t]+2-$fold integral when $t$ is not an integer. Shepp's results are as follows. Let
$$\phi(x):=\frac{1}{\sqrt{2 \pi}} \exp \left( -\frac{x^2}{2} \right) \quad \mbox{and} \quad \phi_{\theta}(x):=\frac{1}{\sqrt{\theta}} \phi \left(\frac{x}{\sqrt{\theta}}\right).$$
When $t=n$ is an integer,
\begin{equation}
\label{integer}
I(t|x) \phi(x)= \int_{\mathcal{D}'} \det\Bigg[\phi(y_i-y_{j+1})\Bigg]_{0 \leq i,j \leq n} dy_2 \cdots dy_{n+1},
\end{equation}
where $y_0=0$, $y_1=|x|$ and $\mathcal{D}':=\{|x|<y_2< \cdots <y_{n+1}\}$. When $t=n+\theta$ where $0 < \theta <1$,
\begin{multline}
\label{noninteger}
I(t|x) \phi(x)=\int_{\mathcal{D}''} \det\Bigg[\phi_{\theta}(x_i-y_i)\Bigg]_{0 \leq i,j \leq n+1} \\
\times \det\Bigg[\phi_{\theta}(y_i-x_{j+1})\Bigg]_{0 \leq i,j \leq n} dx_2 \cdots dx_{n+1}dy_0 \cdots dy_{n+1},
\end{multline}
where $x_0=0$, $x_1=|x|$ and $\mathcal{D}'':=\{|x|<x_2< \cdots <x_{n+1}~\mbox{and}~y_0< \cdots <y_{n+1}\}$.
The distribution of the first passage time $F$ is characterized by 
\begin{equation}
\label{dfpt}
\mathbb{P}(F>t)=\int_{\mathbb{R}} I(t|x)\phi(x)dx,
\end{equation}
where $ I(t|x) \phi(x)$ is given as \eqref{integer} when $t$ is integer and given as \eqref{noninteger} when it is not. 
In particular, 
\begin{align}
\mathbb{P}(F>1) &= \int_{\mathbb{R}} [\Phi(0) \phi(x)- \phi(0)\Phi(x)]dx \notag\\
                               &=\frac{1}{2} -\frac{1}{\pi}, \label{Tbig1}
\end{align}
where $\Phi(x):=\int_{-\infty}^{x}\phi(z)dz$ is the cumulative distribution function of the standard normal distribution. 

In this paper, we study the local structure of the Slepian zero set, i.e. $\{t \in [0,1]; S_t=0\}$, by showing that it is mutually absolutely continuous relative to that of Brownian motion with normally distributed starting point. The main result, which provides a path decomposition of the Slepian process on $[0,1]$, is stated as below.
\begin{theorem}
\label{main}
Let $F:=\inf\{t \geq 0; S_t=0\}$ and $G:=\sup\{t \leq 1; S_t=0\}$. Given the quadruple $(S_0,S_1,F,G)$ with $0<F<G<1$, the Slepian process $(S_t; 0 \leq t \leq 1)$ is decomposed into three conditionally independent pieces:
\begin{itemize}
\item
$(S_t/\sqrt{2};0 \leq t \leq F)$ is Brownian first passage bridge from $(0,S_0/\sqrt{2})$ to $(F,0)$;
\item
$(S_t/\sqrt{2};F \leq t \leq G)$ is Brownian bridge of length $G-F$;
\item
$(|S_t|/\sqrt{2}; G \leq t \leq 1)$ is a three-dimensional Bessel bridge from $(G,0)$ to $(1,|S_1|/\sqrt{2})$.
\end{itemize} 
In addition, the distribution of $(S_0,S_1,F,G)$ with $0<F<G<1$ is given by
\begin{multline}
\label{bigdens}
\mathbb{P}(S_0 \in dx, S_1 \in dy, F \in da, G \in db) =\\ \frac{|xy|}{8 \pi^2 \sqrt{ (b-a)a^3(1-b)^3}}\exp\left(-\frac{x^2}{4a}-\frac{y^2}{4(1-b)}-\frac{(x+y)^2}{4}\right).
\end{multline}
\end{theorem}
\begin{figure}[h]
\begin{center}
\includegraphics[width=0.6 \textwidth]{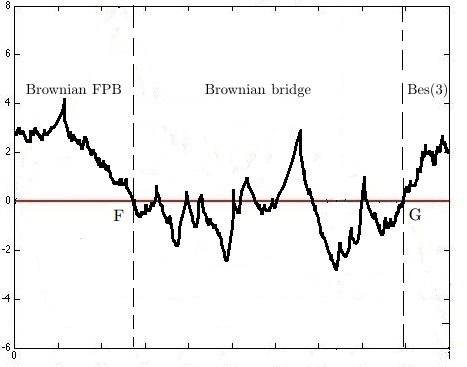}
\end{center}
\caption{Path decomposition of $(S_t/\sqrt{2};0 \leq t \leq 1)$ with $0<F<G<1$.}
\end{figure}

On the event $\{0<F<G<1\}$, the Slepian process is achieved by first creating the quadruple $(S_0,S_1,F,G)$ and then filling in with usual Brownian components. Similarly, on the event $\{F>1\}$, $(S_t/\sqrt{2}; 0 \leq t \leq 1)$ is Brownian bridge from $(0,S_0/\sqrt{2})$ to $(1,S_1/\sqrt{2})$ conditioned not to hit $0$. 

The proof of Theorem \ref{main} is deferred to Section \ref{22}. One method relies on Shepp \cite{SheppGaussian}'s result of the absolute continuity between Gaussian measures, where the Slepian process was proved to be mutually absolutely continuous with respect to some modified Brownian motion on $[0,1]$. As pointed out by Shepp \cite{Shepp}, the absolute continuity fails beyond the unit interval. This is why we  restrict the study of the Slepian zero set to intervals of length $1$. Nevertheless, we have the following conjecture:
\begin{conjecture}
For $t \geq 0$, the Slepian zero set on $[0,t]$, i.e. $\{u \in [0,t]; S_u=0\}$, is mutually absolutely continuous with respect to that of $\{u \in [0,t]; \xi+B_u=0\}$, the zero set of Brownian motion started at $\xi$ with standard normal distribution, $\xi \sim \mathcal{N}(0,1)$,
and $\xi$ independent of $B$.
\end{conjecture}
\section{The Slepian zero set and path decomposition}
\label{2}
In this section, we study the Slepian zero set on $[0,1]$, that is $\{u \in [0,1]; S_u=0\}$. The problem here involves level crossings of a stationary Gaussian process. We refer readers to the surveys of Blake and Lindsey \cite{BL}, Abrahams \cite{Abr}, Kratz \cite{Kratz}, as well as the books of Cram\'{e}r and Leadbetter \cite[Chapter $10$]{CL}, Aza\"{i}s and Wschebor \cite[Chapter $3$]{AW} for further development. 

Berman \cite{Berman} studied general criteria for stationary Gaussian processes to have local times. In particular, he proved that if $(Z_t; t \geq 0)$ is a stationary Gaussian process with covariance $R^Z(t)$ and $1-R^Z(t) \underset{t \rightarrow 0}{\sim} |t|^{\alpha}$ for some $0<\alpha<2$, then $Z$ has local times $(L^x_t; x \in \mathbb{R},t \geq 0 )$ such that for any Borel measurable set $C \subset \mathbb{R}$ and $t \geq 0$, 
$$\int_0^t 1(Z_s \in C) ds =\int_C  L_t^x dx.$$
As discussed in the introduction, the Slepian process has covariance $R^S(t):=\max(1-|t|,0)$, which obviously fits into the above category. See also the survey of Geman and Horowitz \cite{GH} for further development on Gaussian occupation measures. 

Below is the plan for this section:

 In Subsection \ref{21}, we deal with the local absolute continuity between the distribution of the Slepian process and that of Brownian motion with random starting point. This follows some general discussion on the absolute continuity between Gaussian measures by Shepp \cite{SheppGaussian}. 

In Subsection \ref{22}, we give two proofs for the path decomposition of the Slepian process $(S_t; 0 \leq t \leq 1)$, Theorem \ref{main}. There we provide an alternative construction of $(S_t; t \geq 0)$, Proposition \ref{consts}.

In Subsection \ref{23}, we study a {\em Palm-It\^{o} measure} associated to the gaps between Slepian zeros, with comparison to the well-known {\em It\^{o}'s excursion law} \cite{Itoex}.
 \subsection{Local absolute continuity between Slepian zeros and Brownian zeros}
 \label{21}
 As proved by Shepp \cite{SheppGaussian}, for each fixed $t \leq 1$, the distribution of the Slepian process $(S_u; 0 \leq u \leq t)$ is mutually absolutely continuous with respect to that of
\begin{equation}
\label{Btil}
(\widetilde{B}_u:=\sqrt{2}(\xi+B_u); 0 \leq u \leq t),
\end{equation}
where $\xi \sim \mathcal{N}(0,1)$ is independent of $(B_u; u \geq 0)$. The Radon-Nikodym derivative is given by
\begin{equation}
\label{abssw}
\frac{d\mathbb{P}^{\bf S}}{d \mathbb{P}^{\bf \widetilde{W}}}\left(w\right):=\frac{2}{\sqrt{2-t}}\exp\left(\frac{w_0^2}{4}-\frac{(w_0+w_t)^2}{4(2-t)}\right),
\end{equation}
where $\mathbb{P}^{\bf S}$ (resp. $\mathbb{P}^{\bf \widetilde{W}}$) is the distribution of the Slepian process $S$ (resp. the modified Brownian motion $\widetilde{B}$ defined as in \eqref{Btil}) on $\mathcal{C}[0,1]$. As a first application, we compute the density of the first passage time $F$, defined as in \eqref{Ts}, on the unit interval.
\begin{proposition}
\label{prop21}
For $w \in \mathcal{C}[0,1]$, let $F:=\inf\{t \geq 0; w_t=0\}$. Then
\begin{equation}
\label{explicitT}
\mathbb{P}^{\bf S}(F \in da) =\frac{1}{\pi} \sqrt{\frac{2-a}{a}}da \quad \mbox{for}~0 \leq a \leq 1,
\end{equation}
\end{proposition}
\begin{proof} Fix $a \leq 1$. By the change of measure formula \eqref{abssw},
\begin{align}
\mathbb{P}^{\bf S}(F \in da) &=\mathbb{E}^{\bf \widetilde{W}}\left[1(F \in da) \cdot \frac{2}{\sqrt{2-a}}\exp\left(\frac{w_0^2}{4}-\frac{(w_0+w_a)^2}{4(2-a)}\right)\right] \notag\\
                    &=\frac{2}{\sqrt{2-a}} \mathbb{E}^{\bf \widetilde{W}}\left[1(F \in da) \exp\left(\frac{w_0^2}{4}-\frac{w_0^2}{4(2-a)}\right)\right] \notag\\
                    &\label{TS1}=\frac{2}{\sqrt{2-a}} \int_{\mathbb{R}} \frac{1}{\sqrt{4 \pi}} \exp\left(-\frac{x^2}{4}\right) \cdot \mathbb{P}^{\bf \widetilde{W}_x}(F \in da) \exp\left(\frac{x^2}{4}-\frac{x^2}{4(2-a)}\right)dx,
\end{align}
where $\mathbb{P}^{\bf \widetilde{W}_x}$ is the distribution of $\widetilde{B}$ conditioned on $\widetilde{B}_0=x$. It is well-known that
\begin{equation}
\label{TS2}
\mathbb{P}^{\bf \widetilde{W}_x}(F \in da)=\frac{|x|}{\sqrt{4 \pi a^3}}\exp\left(-\frac{x^2}{4a}\right)da.
\end{equation}
Injecting \eqref{TS2} into \eqref{TS1}, we obtain
\begin{align*}
\mathbb{P}^{\bf S}(F \in da) &=\frac{1}{2 \pi \sqrt{(2-a)a^3}} \int_{\mathbb{R}}|x| \exp\left(-\frac{x^2}{2a(2-a)}\right)dx da \\
                    &= \frac{1}{\pi} \sqrt{\frac{2-a}{a}}da.
\end{align*}
\end{proof}
\begin{remark}
As a check, from \eqref{explicitT}, $$\mathbb{P}^{\bf S}(F \leq 1)=\frac{1}{2}+\frac{1}{\pi} \approx 0.82,$$
which agrees with the formula \eqref{Tbig1} derived from the determinantal expressions \eqref{integer}, \eqref{noninteger} and \eqref{dfpt}. Since the absolute continuity relation does not hold when $t>1$, we are not able to derive a simple formula for the density of $F$ on $(1,\infty)$.
\end{remark}

Next we deal with the local absolute continuity between the distribution of Slepian zeros and that of Brownian motion with normally distributed starting point. The result enables us to prove Proposition \ref{scaling}, that is the weak convergence of the discrete first level bridges to the bridge-like process as in \eqref{bridgelike}.
\begin{lemma}
\label{abszero}
For each fixed $t \geq 0$, the distribution of $(S_u; t \leq u \leq t+1)$ is mutually absolutely continuous with respect to that of $(\widetilde{B}_u; t \leq u \le t+1)$ defined as in \eqref{Btil}. The Radon-Nikodym derivative is given by
\begin{equation}
\label{absstilw}
\frac{d \mathbb{P}^{\bf S}}{d \mathbb{P}^{\bf \widetilde{W^t}}}(w)=2 \sqrt{\frac{1+t}{2-t}}\exp\left(\frac{w_0^2}{4(1+t)}-\frac{(w_0+w_1)^2}{4(2-t)}\right),
\end{equation}
where $\mathbb{P}^{\widetilde{\bf W^t}}$ is the distribution of $\widetilde{B}$ on $[t,t+1]$. In particular, the distribution of the Slepian zero set restricted to $[t,t+1]$, i.e. $\{u \in [t,t+1]; S_u=0\}$ is mutually absolutely continuous with respect to that of $\{u \in [t,t+1]; \xi+B_u=0\}$, the zero set of Brownian motion starting at $\xi \sim \mathcal{N}(0,1)$.
\end{lemma}
\begin{proof} It suffices to prove the first part of this lemma. By stationarity of the Slepian process, the distribution of $(S_u; t \leq u \leq t+1)$ is the same as that of $(S_u; 0 \leq u \leq 1)$, which is mutually absolutely continuous relative to $(\widetilde{B}_u; 0 \leq u \leq 1)$ with density given by \eqref{abssw}. Now we conclude by noting that the distribution of $(\widetilde{B}_u; t \leq u \leq t+1)$ and that of $(\widetilde{B}_u; 0 \leq u \leq 1)$ are mutually absolutely continuous, with Radon-Nikodym derivative
$$\frac{d \mathbb{P}^{\widetilde{{\bf W}}}}{d \mathbb{P}^{\widetilde{{\bf W^t}}}}(w):=\sqrt{1+t} \exp \left(-\frac{t w_0^2}{4(1+t)}\right).$$
\end{proof}

As a consequence, all local properties of the Slepian zero set mimic closely those of Brownian motion with normally distributed starting point. In particular, with positive probability, the Slepian process visits the origin on the unit interval. And immediately thereafter, it returns to the origin infinitely often, as does Brownian motion. In addition, it is easy to see that the Radon-Nikodym derivative between the distribution of $\{u \in [0,1]; S_u=0\}$ and that of $\{u \in [0,1]; \xi+B_u=0\}$ is given by
\begin{equation}
\label{RNzero}
\mathbb{E}\left[\frac{d \mathbb{P}^{\bf S}}{d \mathbb{P}^{\bf \widetilde{W}}}(w)\Bigg| \mbox{Proj}(w)\right],
\end{equation}
where $\mbox{Proj}(w):=\{u \in [0,1]; w_u=0\}$ is the zero set of $w \in \mathcal{C}[0,1]$. In the next subsection, we will see how this conditional expectation, as the Radon-Nikodym derivative, can be made explicit by some sufficient statistics.
 \subsection{Path decomposition of the Slepian process on $[0,1]$}
 \label{22}
In this subsection, we investigate further the local structure of the Slepian zero set by proving Theorem \ref{main}. 

From the work of Slepian \cite{Slepian}, we know that given the i.i.d. normally distributed sequence $(S_n:=B_{n+1}-B_{n})_{n \in \mathbb{N}}$, for each $n \in \mathbb{N}$, the process $(S_{n+u}/\sqrt{2}; 0 \leq u \leq 1)$ has the same distribution as Brownian bridge from $S_n/\sqrt{2}$ to $S_{n+1}/\sqrt{2}$. For $n \in \mathbb{N}$ and $k \geq 2$, the processes $(S_{n+u}/\sqrt{2}; 0 \leq u \leq 1)$ and $(S_{n+k+u}/\sqrt{2}; 0 \leq u \leq 1)$ are independent. 
However, the consecutive bridges $(S_{n+u}/\sqrt{2}; 0 \leq u \leq 1)$ and $(S_{n+1+u}/\sqrt{2}; 0 \leq u \leq 1)$ for $n \in \mathbb{N}$, are correlated. The correlation is inferred from the following construction of the Slepian process.
\begin{proposition}
\label{consts}
Let $(Z_n)_{n \in \mathbb{N}}$ be a sequence of i.i.d. $\mathcal{N}(0,1)$-distributed random variables, and $(b^n_t;0 \leq t \leq 1)_{n \in \mathbb{N}}$ be a sequence of i.i.d. standard Brownian bridges independent of $(Z_n)_{n \in \mathbb{N}}$. Define a continuous-time process $(Z_t; t \geq 0)$ as
\begin{equation}
\label{cstPT}
Z_t:=b^{n+1}_{t-n}-b^{n}_{t-n}+(n+1-t)Z_n+(t-n)Z_{n+1} \quad \mbox{for}~n \leq t <n+1,~n \in \mathbb{N}.
\end{equation}
Then $(Z_t; t \geq 0)$ has the same distribution as the Slepian process $(S_t;t \geq 0)$.
\end{proposition}
\begin{proof} Note that $(Z_t; t \geq 0)$ and $(S_t; t \geq 0)$ are centered Gaussian processes. It suffices to show that $(Z_t; t \geq 0)$ and $(S_t; t \geq 0)$ have the same covariance function. Let $t_2 \geq t_1 \geq 0$. Recall that $\mathbb{E}[S_{t_1}S_{t_2}]=\max(1-t_2+t_1,0)$. From the construction \eqref{cstPT} of $(Z_t; t \geq 0)$, we know that $Z_{t_1}$ and $Z_{t_2}$ are independent if $t_1 \in [n,n+1)$ and $t_2 \geq n+2$ for some $n \in \mathbb{N}$. In this case, $\mathbb{E}[Z_{t_1}Z_{t_2}]=0$. The other cases are:
\begin{itemize}
\item
$t_1,t_2 \in [n,n+1)$ for some $n \in \mathbb{N}$. Then
\begin{align*}
\mathbb{E}[Z_{t_1}Z_{t_2}] &=\mathbb{E}[b_{t_1-n}^{n+1}b_{t_2-n}^{n+1}] + \mathbb{E}[b_{t_1-n}^nb_{t_2-n}^n] \\
                                                & \quad \quad \quad \quad \quad +(n+1-t_1)(n+1-t_2)\mathbb{E}Z_n^2+(t_1-n)(t_2-n)\mathbb{E}Z_{n+1}^2 \\
                                                &=2(t_1-n)(n+1-t_2) + (n+1-t_1)(n+1-t_2) + (t_1-n)(t_2-n) \\
                                                &=1-t_2+t_1.
\end{align*}
\item
$t_1 \in [n,n+1)$ and $t_2 \in [n+1,n+2)$ for some $n \in \mathbb{N}$. Then
$$\mathbb{E}[Z_{t_1}Z_{t_2}]=-\mathbb{E}[b^{n+1}_{t_1-n}b^{n+1}_{t_2-n-1}]+(t_1-n)(n+2-t_2) \mathbb{E}Z_{n+1}^2.$$
When $t_2-t_1 \geq 1$, we obtain:
$$\mathbb{E}[Z_{t_1}Z_{t_2}]=-(t_1-n)(n+2-t_2)+(t_1-n)(n+2-t_2) =0.$$
When $t_2-t_1 <1$, we obtain:
$$\mathbb{E}[Z_{t_1}Z_{t_2}]=-(n+1-t_1)(t_2-n-1)+(t_1-n)(n+2-t_2) =1-t_2+t_1.$$
\end{itemize}
Putting all pieces together, we have $\mathbb{E}[Z_{t_1}Z_{t_2}]=\max(1-t_2+t_1,0)=\mathbb{E}[S_{t_1}S_{t_2}]$.  
\end{proof}
\begin{remark}
In particular, the proposition shows that given the triple of i.i.d. standard normal variables $(S_n, S_{n+1}, S_{n+2} )$, the two standard Brownian bridges derived from $(S_{n+u}; 0 \leq u \leq 1)$ and $(S_{n+1+u}; 0 \leq u \leq 1)$ by subtracting off the lines between endpoints:
$$\left(\frac{S_{n+u}-(1-u)S_n-uS_{n+1}}{\sqrt{2}}; 0 \leq u \leq 1 \right)~\mbox{and} ~ \left(\frac{S_{n+1+u}-(1-u)S_{n+1}-uS_{n+2}}{\sqrt{2}}; 0 \leq u \leq 1 \right)$$
are not conditionally independent.
\end{remark}

For $w \in \mathcal{C}[0,1]$, let $F:=\inf\{t \geq 0; w_t=0\}$ be the time of first hit to $0$, and $G:=\sup\{t \leq 1; w_t=0\}$ be the time of last exit from $0$ on the unit interval. From Proposition \ref{consts}, the Slepian process on $[0,1]$ can be constructed by first picking independently $S_0,S_1 \sim \mathcal{N}(0,1)$, and then filling in a $\sqrt{2}$-Brownian bridge from $S_0$ to $S_1$. This bridge construction provides a proof of Theorem \ref{main}. 
\begin{proof}[Proof of Theorem \ref{main}] The first part of the statement is quite straightforward from Proposition \ref{consts} and the discussion above. To finish the proof, we compute the $\mathbb{P}^{\bf S}$-joint distribution of the quadruple $(w_0,w_1,F,G)$ on the event $\{0<F<G<1\}$.
\begin{equation}
\label{alt}
\mathbb{P}^{\bf S}(w_0 \in dx,w_1 \in dy,F \in da,G \in db) =\frac{dxdy}{2 \pi} \exp\left(-\frac{x^2+y^2}{2}\right) \mathbb{P}^{\bf \widetilde{W}_{x \rightarrow y}}(F \in da,G \in db),
\end{equation}
where $\mathbb{P}^{\bf \widetilde{W}_{x \rightarrow y}}$ is the distribution of $\widetilde{B}$ defined as in \eqref{Btil}, conditioned to start at $x$ and end at $y$. In addition,
\begin{align}
\label{TGJ}
&~~ \quad \mathbb{P}^{\bf \widetilde{W}_{x \rightarrow y}}(F \in da,G \in db) \notag\\
&=\frac{|x|}{\sqrt{4 \pi(1-a)a^3}} \exp\left(-\frac{y^2}{4(1-a)}-\frac{x^2}{4a}+\frac{(x-y)^2}{4}\right) \cdot \mathbb{P}^{\bf \widetilde{W}_{x \rightarrow y}}(G \in db|F \in da) \notag\\
&=\frac{|x|}{\sqrt{4 \pi(1-a)a^3}} \exp\left(-\frac{y^2}{4(1-a)}-\frac{x^2}{4a}+\frac{(x-y)^2}{4}\right) \notag\\
& \quad \quad \quad \quad \quad \quad \quad \quad \quad  \quad \quad \quad \quad \cdot \frac{|y|\sqrt{1-a}}{\sqrt{4 \pi (b-a)(1-b)^3}} \exp\left(\frac{y^2}{4(1-a)}-\frac{y^2}{4(1-b)}\right) \notag\\
&=\frac{|xy|}{4 \pi \sqrt{(b-a)a^3(1-b)^3}} \exp\left(-\frac{x^2}{4a}-\frac{y^2}{4(1-b)}+\frac{(x-y)^2}{4}\right).
\end{align}
Injecting \eqref{TGJ} into \eqref{alt}, we obtain the formula \eqref{bigdens}. 
\end{proof}

By integrating over $S_0 \in dx$ and $S_1 \in dy$ in the formula \eqref{bigdens}, we get:
\begin{multline}
\label{220}
\mathbb{P}(F \in da, G \in db) \\ =\frac{2}{\pi \sqrt{b-a}}\left[\frac{1}{(2+a-b)\sqrt{a(1-b)}} +\frac{1}{\sqrt{(2+a-b)^3}} \arctan \sqrt{\frac{a(1-b)}{2+a-b}}\right],
\end{multline}
and integrating further \eqref{220} over $G \in db$, we obtain:
$$\mathbb{P}(F \in da) =\frac{1}{\pi}\sqrt{\frac{2-a}{a}}da \quad \mbox{for}~0<a<1,$$
which agrees with the formula \eqref{explicitT} found in Proposition \ref{21}.

In the rest of the subsection, we give yet another proof of Theorem \ref{main}. We start by deriving the formula \eqref{bigdens} from the absolute continuity relation \eqref{abssw}.
\begin{proof}[Proof of \eqref{bigdens} by \eqref{abssw}] We first compute the $\mathbb{P}^{\bf \widetilde{W}}$-joint distribution of $(w_0,w_1,F,G)$, where $\mathbb{P}^{\bf \widetilde{W}}$ is the distribution of $\widetilde{B}$ on $[0,1]$ defined as in \eqref{Btil}. 
\begin{align}
&~~\quad  \mathbb{P}^{\bf \widetilde{W}}(w_0 \in dx,w_1 \in dy,F \in da,G \in db) \notag\\
&=\frac{dx}{\sqrt{4 \pi}} \exp\left(-\frac{x^2}{4}\right) \cdot \mathbb{P}^{\bf \widetilde{W}}(w_1 \in dy, F \in da, G \in db|w_0 \in dx) \notag\\
&=\frac{dx}{\sqrt{4 \pi}} \exp\left(-\frac{x^2}{4}\right) \cdot \frac{|x|da}{\sqrt{4 \pi a^3}}\exp\left(-\frac{x^2}{4a}\right) \cdot \mathbb{P}^{\bf \widetilde{W}}(w_1 \in dy, G \in db|w_0 \in dx, F \in da) \notag\\
&\label{213}=\frac{dx}{\sqrt{4 \pi}} \exp\left(-\frac{x^2}{4}\right) \cdot \frac{|x|da}{\sqrt{4 \pi a^3}}\exp\left(-\frac{x^2}{4a}\right) \cdot \frac{|y|dydb}{4 \pi \sqrt{(b-a)(1-b)^3}}\exp\left(-\frac{y^2}{4(1-b)}\right) \\
&\label{214}=\frac{|xy|}{16 \pi^2 \sqrt{(b-a)a^3(1-b)^3}}\exp\left(-\frac{x^2}{4}-\frac{x^2}{4a}-\frac{y^2}{4(1-b)}\right)\,dx\,dy\,da\,db,
\end{align}
where \eqref{213} can be read from Revuz and Yor \cite[Exercise $3.23$, Chapter III]{RY}. Now \eqref{214} combined with \eqref{abssw} yields the desired result. 
\end{proof}

We need the following elementary result regarding the change of measures.
\begin{lemma}
\label{condcom}
Assume that $\mathbb{P}$ and $\mathbb{Q}$ are two probability measures on $(\Omega,\mathcal{F})$ such that
$$\frac{d \mathbb{Q}}{d \mathbb{P}}(w):=f(Z),$$
where $Z:=Z(w)$ is a random element and $f(Z)$ is the Radon-Nikodym derivative of $\mathbb{Q}$ with respect to $\mathbb{P}$. Futhermore,
\begin{enumerate}
\item
Let $A \in \sigma(Z)$ be an event determined by $Z$, with $\mathbb{P}(A)>0$ and $\mathbb{Q}(A)>0$;
\item
Let $Y$ be another random element such that under $\mathbb{P}$, $Y$ is independent of $Z$ given $A$ (such random element $Y$ need only be defined conditional on $A$).
\end{enumerate}
Then the $\mathbb{Q}$-distribution of $Y$ given $A$ is the same as the $\mathbb{P}$-distribution of $Y$ given $A$. And under $\mathbb{Q}$, $Y$ is independent of $Z$ given $A$.
\end{lemma}
\begin{proof}Take $g,h: (\Omega,\mathcal{F}) \rightarrow (\mathbb{R},\mathcal{B}(\mathbb{R}))$ two bounded measurable functions. First note that
\begin{align}
\mathbb{E}^{\mathbb{Q}}[g(Y)|A] &=\frac{\mathbb{P}(A)}{\mathbb{Q}(A)} \mathbb{E}^{\mathbb{P}}[g(Y)f(Z)|A] \notag\\
                                                          &\label{exp1}= \frac{\mathbb{P}(A)}{\mathbb{Q}(A)}  \mathbb{E}^{\mathbb{P}}[f(Z)|A] \cdot \mathbb{E}^{\mathbb{P}}[g(Y)|A] \\
                                                          &\label{equalpq}=\mathbb{E}^{\mathbb{P}}[g(Y)|A],
\end{align}
where \eqref{exp1} is due to the $\mathbb{P}$-conditional independence of $Y$ and $Z$ given $A$. In addition, 
\begin{align}
\mathbb{E}^{\mathbb{Q}}[g(Y)h(Z)|A] &=\frac{\mathbb{P}(A)}{\mathbb{Q}(A)} \mathbb{E}^{\mathbb{P}}[g(Y)h(Z)f(Z)|A] \notag\\
                                                                &\label{exp2}=\frac{\mathbb{P}(A)}{\mathbb{Q}(A)} \mathbb{E}^{\mathbb{P}}[h(Z)f(Z)|A] \cdot  \mathbb{E}^{\mathbb{P}}[g(Y)|A] \\
                                                                &\label{fin}= \mathbb{E}^{\mathbb{Q}}[h(Z)|A] \cdot \mathbb{E}^{\mathbb{Q}}[g(Y)|A],
\end{align}
where \eqref{exp2} is again due to the $\mathbb{P}$-conditional independence of $Y$ and $Z$ given $A$, and \eqref{fin} follows readily from \eqref{equalpq}.  \end{proof}

\begin{proof}[Alternative proof of Theorem \ref{main}] We borrow the notations from Lemma \ref{condcom} in our setting: $\mathbb{P}:=\mathbb{P}^{\bf \widetilde{W}}$, $\mathbb{Q}:=\mathbb{P}^{\bf S}$, $Z:=(w_0,w_1,F,G)$ and $A:=\{0<F<G<1\}$. Conditional on $A$, define $Y^{(2)}$ to be the scaled bridge on $[F,G]$, that is
$$Y^{(2)}_u:=\frac{w_{F+u(G-F)}}{\sqrt{G-F}} \quad \mbox{for}~0 \leq u \leq 1.$$
It is well-known that under $\mathbb{P}^{\bf \widetilde{W}}$ and on the event $\{0<F<G<1\}$, $(Y^{(2)}_u/\sqrt{2}; 0 \leq u \leq 1)$ is standard Brownian bridge, independent of $(w_0,w_1,F,G)$, see e.g. L\'evy \cite{Levy} or Revuz and Yor \cite[Exercise $3.8$, Chapter XII]{RY}. Then by Lemma \ref{condcom}, under $\mathbb{P}^{\bf S}$ and on the event $\{0<F<G<1\}$, $(Y_u^{(2)}/\sqrt{2};0 \leq u \leq 1)$ is also standard Brownian bridge, independent of $(w_0,w_1,F,G)$. In addition, define $Y^{(1)}:=(w_u;0 \leq u \leq F )$ and $Y^{(3)}:=(w_u; G \leq u \leq 1)$. Similarly, under $\mathbb{P}^{\bf \widetilde{W}}$ and on the event $\{0<F<G<1\}$,
\begin{itemize} 
\item $Y^{(1)}/\sqrt{2}$ is Brownian first passage bridge from $(0,w_0/\sqrt{2})$ to $(F,0)$, see e.g. Bertoin et al \cite{BCP};
\item $|Y^{(3)}|/\sqrt{2}$ is reversed Brownian first passage bridge from $(1,|w_1|/\sqrt{2})$ to $(G,0)$, that is the three-dimensional Bessel bridge from $(G,0)$ to $(1,|w_1|/\sqrt{2})$, see e.g. Biane and Yor \cite{BY}.
\end{itemize}
Moreover, these two processes are conditionally independent given $(w_0,w_1,F,G)$. It suffices to apply again Lemma \ref{condcom} to conclude. 
\end{proof}

In view of the Brownian characteristics,it would be interesting to find a construction of the conditioned Slepian process $(S_t/\sqrt{2}; 0 \leq t \leq 1)$ with $\{0<F<G<1\}$ by some path transformation of standard Brownian motion/bridge. We leave the interpretation open for future investigation.
 \subsection{A Palm-It\^{o} measure related to Slepian zeros}
 \label{23}
To capture the structure of the Slepian zero set, an alternative way is to study the Slepian excursions between consecutive zeros. Let $E$ be the space of excursions defined by
$$E:=\{\epsilon \in \mathcal{C}[0,\infty); \epsilon_0=0~\mbox{and}~\epsilon_t=0~\mbox{for all}~t \geq \zeta(\epsilon) \in ]0,\infty[\},$$
where $\zeta(\epsilon):=\inf\{t > 0; \epsilon_t = 0\}$ is the lifetime of the excursion $\epsilon \in E$. Following Pitman \cite{Pitmanexcursion}, the gaps between zeros of a process $(Z_t;t \geq 0)$ with $\sigma$-finite invariant measure can be described by a {\em Palm-It\^{o} measure} ${\bf n}^Z$, defined on the space of excursions $E$ as
\begin{equation*}
{\bf n}^Z(d\epsilon):=  \mathbb{E}\#\{0<t<1; Z_t=0~\mbox{and}~ e(t)\in d\epsilon\},
\end{equation*}
where $e(t)$ is the excursion starting at time $t>0$ in the process $Z$. The following result of {\em last exit decomposition} for stationary processes is read from Pitman \cite[Theorem $1$(iii)]{Pitmanexcursion}.
\begin{theorem} \cite{Pitmanexcursion} \label{Pitex}
Let $\mathbb{P}^{\bf Z}$ govern a stationary process $(Z_t; t \geq 0)$ with $\sigma$-finite invariant measure. For $w \in \mathcal{C}[0,1]$, let
\begin{equation}
\label{gtt}
G_t:=\sup\{u \leq t; w_u=0\},
\end{equation} 
be the last exit time from $0$ before time $t$, and $e(G_t)$ be the excursion straddling time $t>0$ in the path. Then
\begin{equation}
\label{lasts}
\mathbb{P}^{\bf Z}(t-G_t \in da, e(G_t) \in d\epsilon)=da 1(\zeta(\epsilon)>a) {\bf n}^Z(d\epsilon),
\end{equation}
\end{theorem}
\begin{remark} ~{\em 
\begin{enumerate}
\item
Theorem \ref{Pitex} extends a result of Bismut \cite{Bismut}, where $Z$ is Brownian motion with invariant Lebesgue measure. In this case, the Palm-It\^{o} measure $\textbf{n}^Z$ is just It\^{o}'s excursion law $\textbf{n}$. 
\item
There is an analog of last exit decomposition \eqref{lasts} for standard Brownian motion. Let $\mathbb{P}^{\bf W}$ be Wiener measure on $\mathcal{C}[0,\infty)$, then
\begin{equation*}
\mathbb{P}^{\bf W}(t-G_t \in da, e(G_t) \in d\epsilon)=da \frac{1}{\sqrt{2 \pi(t-a)}} 1(\zeta(\epsilon)>a){\bf n}(d\epsilon),
\end{equation*}
where ${\bf n}$ is It\^{o}'s excursion law. The result is deduced from Getoor and Sharpe \cite{GSde}, who gave a last exit decomposition for general Markov processes.
\end{enumerate}
}
\end{remark}

Now we apply Theorem \ref{Pitex} to the Slepian process $(S_t; t\geq 0)$. Let $\mathbb{P}^{\textbf{S}}$ be the distribution of the Slepian process, we have:
\begin{equation}
\label{215}
\mathbb{P}^{\textbf{S}}(t-G_t \in da, e(G_t)\in d\epsilon)=da 1(\zeta(\epsilon)>a)\textbf{n}^S(d\epsilon),
\end{equation}
where
\begin{equation}
\label{nsit}
\textbf{n}^S(d\epsilon):=\mathbb{E}\#\{0<t<1; S_t=0~\mbox{and}~e(t) \in d\epsilon\}.
\end{equation}
\quad As shown in the following lemma, the last exit time $G_t$ is closely related to the first passage time $F$ defined as in \eqref{Ts}. Here we adopt the convention that $\sup \emptyset:=0$.
\begin{lemma}
\label{gT}
Let $t > 0$. Under $\mathbb{P}^{\textbf{S}}$, $t-G_t$ has the same distribution as $F \cdot 1(F \leq t )+t \cdot 1(F>t)$, where $G_t$ is defined by \eqref{gtt} and $F$ by \eqref{Ts}.
\end{lemma}
\begin{proof} It suffices to observe that $(S_u; 0 \leq u \leq t)$ has the same distribution as $(S_{t-u}; 0 \leq u \leq t)$. This is clear from the covariance function of the Slepian process. 
\end{proof}
\begin{proposition}
\label{nct}
For $a>0$,
\begin{equation}
\label{levy}
{\bf n}^{S}(\zeta>a) = \mathbb{P}(F \in da)/da,
\end{equation}
where ${\bf n}^{S}$ is defined as in \eqref{nsit} and $F$ is defined as in \eqref{Ts}. In particular,
\begin{equation}
\label{explicitl}
{\bf n}^{S}(\zeta>a) = \frac{1}{\pi}\sqrt{\frac{2-a}{a}} \quad \mbox{for}~0<a<1.
\end{equation}
\end{proposition}
\begin{proof} It follows from \eqref{lasts} that
$${\bf n}^{S}(\zeta>a)da =\mathbb{P}(t-G_t \in da).$$ 
According to Lemma \ref{gT},
$$\mathbb{P}(t-G_t \in da)=\mathbb{P}(F \in da) \quad \mbox{for}~t>a.$$
Then \eqref{levy} is a direct consequence of the above two observations. Combining with the formula \eqref{explicitT}, we derive further \eqref{explicitl}. 
\end{proof}

From the {\em Palm-L\'evy measure} \eqref{explicitl}, we can see how the Slepian zero set restricted to $[0,1]$ differs from a plain Brownian zero set, where {\em It\^o's excursion law} is given by
$${\bf n}(\zeta>a)=\sqrt{\frac{2}{\pi a}} \quad \mbox{for}~a>0.$$
As expected, $\textbf{n}^S(\zeta>a)$ and $\textbf{n}(\zeta>a)$ have asymptotically equivalent tails $a^{-\frac{1}{2}}$ when $a \rightarrow 0^{+}$. Observe a constant factor $\sqrt{\pi}$ between them. This is because the invariant Lebesgue measure of reflected Brownian motion is $\sigma$-finite and there is no canonical normalization. We also refer readers to Pitman and Yor \cite[Section $2$]{PYl} for the Palm-L\'{e}vy measure of the gaps between zeros of squared Ornstein-Uhlenbeck processes.
\section{Brownian bridge embedded in Brownian motion}
\label{3}
In this section, we prove Theorem \ref{mainbis}, that is embedding Brownian bridge $(b^0_u; 0 \leq u \leq 1)$ into Brownian motion $(B_t; t \geq 0)$ by a random translation of origin in spacetime. The problem of embedding continuous paths into Brownian motion was broadly discussed in Pitman and Tang \cite{PTpattern}, to which we refer readers for a bird's-eye view.

 Recall that $(X_t; t \geq 0)$ is the moving-window process associated to Brownian motion defined as in \eqref{mmm}. We aim to find a random time $T \geq 0$ such that $X_T$ has the same distribution as $b^0$. A general result of Rost \cite{Rost2} implies that such a randomized stopping time $T \geq 0$ exists (relative to the filtration of the moving window process,
so $T+1$ would be a randomized stopping time in the Brownian filtration) if and only if
\begin{equation}
\label{crfs}
\lim_{\alpha \rightarrow 0} \sup_{1 \geq g \in \mathcal{S}^{\alpha}} \left(\int g d\mathbb{P}^{\bf W^0} -\int g d \mathbb{P}^{\bf W}\right)=0,
\end{equation}
where $\mathbb{P}^{\bf W}$ is Wiener measure on $\mathcal{C}_0[0,1]$ and $\mathbb{P}^{\bf W^0} $ is Wiener measure pinned to $0$ at time $1$, that is the distribution of Brownian bridge $b^0$. For $\alpha>0$, $\mathcal{S}^{\alpha}$ is the set of $\alpha-$excessive functions, see e.g. the book of Sharpe \cite{Sharpebook} for background. However, the criterion \eqref{crfs} is difficult to check since $\mathbb{P}^{\bf W^0} $ is singular with respect to $\mathbb{P}^{\bf W}$.

We work around the problem in another way, which relies heavily on Palm theory of stationary random measures. Such theory has been successfully developed by the Scandinavian probability school in the last few decades. The book of Thorisson \cite{Thbook} records much of this important work. For technical purposes, we introduce a two-sided Brownian motion $(\widehat{B}_t; t \in \mathbb{R})$, and let
\begin{equation}
\label{tsmw}
\widehat{X}_t:=(\widehat{B}_{t+u}-\widehat{B}_t; 0 \leq u \leq 1) \quad \mbox{for}~t \in \mathbb{R},
\end{equation}
be the moving-window process associated to the two-sided Brownian motion $\widehat{B}$. Note that $(\widehat{X}_t; t \in \mathbb{R})$ is a stationary Markov process with state space $(\mathcal{C}_0[0,1], \mathcal{B})$. Alternatively, $\widehat{X}$ can be viewed as a random element in the space $\mathcal{C}_0[0,1]^{\mathbb{R}}$, to which we assign the metric $\rho$ by 
\begin{equation}
\label{metric}
\rho(x, y):=\sum_{n=0}^{\infty} \frac{1}{2^n} \min \left(\sup_{-n \leq t \leq n} ||x_t-y_t||_{\infty},1 \right) \quad \mbox{for}~x,y \in \mathcal{C}_0[0,1]^{\mathbb{R}} 
\end{equation}
where $\mathcal{C}_0[0,1]$ is equipped with the sup-norm $||\cdot||_{\infty}$. 

Below is the plan for this section:

In Subsection \ref{31}, we provide background on Palm theory of stationary random measures. We define a notion of local times of the $\mathcal{C}_0[0,1]-$valued process $\widehat{X}$ by weak approximation. Furthermore, we show that the $0$-marginal of the Palm measure of local times is Brownian bridge.

 In Subsection \ref{32}, we derive from a result of Last and Thorisson \cite{LT2014} that the Palm probability measure of a jointly stationary random measure associated to $\widehat{X}$ can be obtained by a random time-shift of $\widehat{X}$ itself. In particular, there exists a random time $\widehat{T} \in \mathbb{R}$ such that $\widehat{X}_{\widehat{T}}$ has the same distribution as $(b_u^0; 0 \leq u \leq 1)$.

In Subsection \ref{33}, we prove that if some distribution on $\mathcal{C}_0[0,1]$ can be achieved in the moving-window process $\widehat{X}$ associated to two-sided Brownian motion, then we are able to construct a random time $T \geq 0$ such that $\widehat{X}_T$ has that desired distribution. Theorem \ref{mainbis} follows immediately from the above observations.

 In Subsection \ref{34}, after presenting some results of Last et al. \cite{LMT}, we construct a random time $T \geq 0$ such that $\widehat{X}_T$ has the same distribution as $(b^0_u; 0 \leq u \leq 1)$. The construction also makes use of the local times defined in Subsection \ref{31}. The argument is due to Hermann Thorisson.
\subsection{Local times of $\widehat{X}$ and its Palm measure}
\label{31}
In this subsection, we present background on Palm theory of stationary random measures. To begin with, $(\Omega,\mathcal{F},\mathbb{P})$ is a generic probability space on which random elements are defined.

Let $(Z_t; t \in \mathbb{R}) \in E^{\mathbb{R}}$ be a continuous-time process with a measurable state space $(E, \mathcal{E})$. We further assume that the process $Z$ is path-measurable, that is $E^{\mathbb{R}} \times \mathbb{R}\ni (Z,t) \rightarrow Z_t \in E$ is measurable for all $t \in \mathbb{R}$. See e.g. Appendix of Thorisson \cite{Th92} for more on path-measurability. Let $\xi$ be a random $\sigma-$finite measure on $\mathbb{R}$.

Assume that the pair $(Z,\xi)$ is jointly stationary, that is
\begin{equation}
\label{jsdef}
\theta_s(Z,\xi) \stackrel{(d)}{=} (Z,\xi) \quad \mbox{for all}~s \in \mathbb{R},
\end{equation}
where $\theta_sZ:=(Z_{s+t}; t \in \mathbb{R})$ and $\theta_s \xi(\cdot):=\xi(\cdot+s)$ are usual time-shift operations. Then the {\em Palm measure} $\mathbb{P}_{Z, \xi}$ of the jointly stationary pair $(Z, \xi)$ is defined as follows: for $f: E^{\mathbb{R}} \times \mathcal{M}(\mathbb{R}) \rightarrow \mathbb{R}$ bounded measurable,
$$\mathbb{P}_{Z,\xi}f:=\mathbb{E}\int_0^1 f(\theta_t(Z,\xi))\xi(dt).$$
In the rest of the paper, we only need to care about the $E^{\mathbb{R}}$-marginal of $\mathbb{P}_{Z,\xi}$. That is, for $f:E^{\mathbb{R}} \rightarrow \mathbb{R}$ bounded measurable,
\begin{equation}
\label{palm}
\mathbb{P}_{\xi}f:=\mathbb{E} \int_0^1 f(\theta_tZ)\xi(dt).
\end{equation}
By abuse of language, we call $\mathbb{P}_{\xi}$ defined by \eqref{palm} the {\em Palm measure} of the stationary random measure $\xi$. Thus, $\mathbb{P}_{\xi}$ is a $\sigma-$finite measure on $E^{\mathbb{R}}$. If $\mathbb{P}_{\xi}1=\mathbb{E}\xi[0,1)<\infty$, then the normalized measure $\mathbb{P}_{\xi}/\mathbb{P}_{\xi}1$ is called the {\em Palm probability measure} of $\xi$. So far most of the results have been established for $\mathbb{P}_{Z,\xi}$, but they still hold for the marginal $\mathbb{P}_{\xi}$. We refer readers to Kallenberg \cite[Chapter $11$]{Kallenberg}, Thorisson \cite[Chapter $8$]{Thbook}, Last \cite{Last2008, Last2010} and the thesis of Gentner \cite{Gthesis} for further development on Palm versions of stationary random measures.

In the sequel, we adapt  our problem setting to the above abstract framework. We take the state space $E:=\mathcal{C}_0[0,1]$ equipped with its Borel $\sigma-$field $\mathcal{B}$. Recall that $(\widehat{X}_t; t \in \mathbb{R})$, the moving-window process defined as in \eqref{tsmw}, is a random element in the metric space $(\mathcal{C}_0[0,1]^{\mathbb{R}}, \rho)$ with the distance $\rho$ defined by \eqref{metric}.

For a Borel measurable set $C \subset \mathbb{R}$, let
$$\mathcal{BR}^C:=\{w \in \mathcal{C}_0[0,1];w(1) \in C\}$$
be the set of bridge paths with endpoint in $C$. By stationarity of  $(\widehat{X}_t; t \in \mathbb{R})$, for each fixed $t \in \mathbb{R}$, $\{u \in [t,t+1]; \widehat{X}_u \in \mathcal{BR}^0\}-\{t\}$ has the same distribution as $\{u \in [0,1]; \widehat{X}_u \in \mathcal{BR}^0\}$, which is mutually absolutely continuous relative to that of $\{u \in [0,1]; \widetilde{B}_u=0\}$ by Lemma \ref{abszero}. Here $\widetilde{B}$ is the modified Brownian motion as in \eqref{Btil}. Inspired from the notion of Brownian local times, we define a random $\sigma-$finite measure $\Gamma$ on $\mathbb{R}$ as follows: for $n \in \mathbb{N}$ and $C \subset \mathbb{R}$,
\begin{equation}
\label{localtime}
\Gamma([-n,n] \cap C):=\lim_{\epsilon \rightarrow 0} \sqrt{\frac{\pi}{2}}\frac{1}{\epsilon} \int_{[-n,n] \cap C}1(\widehat{X}_u \in \mathcal{BR}^{[-\epsilon,\epsilon]})du.
\end{equation}

Let us justify that the random measure $\Gamma$ as in \eqref{localtime} is well-defined. Write $C=\cup_{k \in \mathbb{Z}}C_k$ where $C_k:=C \cap [k,k+1]$. We want to show that for each $k \in \mathbb{Z}$,
$$\lim_{\epsilon \rightarrow 0} \frac{1}{\epsilon}\int_{C_k} 1(\widehat{X}_u \in \mathcal{BR}^{[-\epsilon,\epsilon]})du \quad \mbox{is well-defined almost surely}.$$
The following lemma is quite straightforward, the proof of which is omitted.
\begin{lemma}
\label{stob}
Assume that two random sequences $(Y_{\epsilon})_{\epsilon>0}$ and $(Y'_{\epsilon})_{\epsilon>0}$ with a measurable state space $(G,\mathcal{G})$ have the same distribution. If $f: G \rightarrow \mathbb{R}$ is a measurable function satisfying that $f(Y_{\epsilon})$ converges almost surely as $\epsilon \rightarrow 0$, then $f(Y'_{\epsilon})$ converges almost surely as $\epsilon \rightarrow 0$
\end{lemma}

Observe that for each fixed $k \in \mathbb{Z}$, $\{u \in [k,k+1]; \widehat{X}_u \in \mathcal{BR}^{[-\epsilon,\epsilon]}\}$ has the same distribution as $\{u \in [0,1]; S_u \in [-\epsilon,\epsilon]\}+\{k\}$ where $(S_u; u \geq 0)$ is the Slepian process. By Lemma \ref{stob}, it suffices to prove that for each Borel measurable set $C' \subset [0,1]$,
$$\lim_{\epsilon \rightarrow 0} \frac{1}{\epsilon}\int_{C'} 1(S_u \in [-\epsilon,\epsilon])du \quad \mbox{is well-defined almost surely}.$$
And this is quite clear from the path decomposition of the Slepian process on $[0,1]$, Theorem \ref{main}. We refer readers to L\'{e}vy \cite{Levybook}, and  Revuz and Yor \cite[Chapter VI]{RY} for the existence of Brownian local times by approximation. Now for $n \in \mathbb{N}$,
$$\frac{1}{\epsilon} \int_{[-n,n] \cap C}1(\widehat{X}_s \in \mathcal{BR}^{[-\epsilon,\epsilon]})ds = \sum_{k=-n}^{n-1}\frac{1}{\epsilon} \int_{C_k}1(\widehat{X}_s \in \mathcal{BR}^{[-\epsilon,\epsilon]})ds$$
converges almost surely as $\epsilon \rightarrow 0$. 

The random measure $\Gamma$ defined by \eqref{localtime} can be interpreted as the local times of the moving-window process $\widehat{X}$ at the level $\mathcal{BR}^0$. Note that the pair $(\widehat{X}, \Gamma)$ is jointly stationary in the sense of \eqref{jsdef}. Next, we compute explicitly the $0$-marginal of the Palm measure of the local times $\Gamma$:   
\begin{proposition}
\label{Revuz}
 Let $\Pi_0: \mathcal{C}_0[0,1]^{\mathbb{R}} \ni w \rightarrow w_0 \in \mathcal{C}_0[0,1]$ be the $0-$marginal projection. Then the image by $\Pi_0$ of the {\em Palm probability measure} of $\Gamma$ as in \eqref{localtime} is 
$$\mathbb{P}_{\Gamma} \circ \Pi_0^{-1}=\mathbb{P}^{\bf W^0},$$
where $\mathbb{P}^{\bf W^0}$ is Wiener measure pinned to $0$ at time $1$, that is the distribution of standard Brownian bridge.
\end{proposition}
\begin{proof} Take $f: \mathcal{C}_0[0,1] \rightarrow \mathbb{R}$ bounded continuous. By injecting \eqref{localtime} into \eqref{palm}, we obtain:
\begin{align}
\mathbb{P}_{\Gamma} \circ \Pi_0^{-1}f&= \lim_{\epsilon \rightarrow 0}\sqrt{\frac{\pi}{2}}\frac{1}{\epsilon} \int_0^1 \mathbb{E}[f(\widehat{X}_t) 1(\widehat{X}_t \in \mathcal{BR}^{[-\epsilon,\epsilon]})]dt \notag\\
                                                 &\label{sta}=\lim_{\epsilon \rightarrow 0}\sqrt{\frac{\pi}{2}}\frac{1}{\epsilon} \mathbb{E}[f(\widehat{X}_0) 1(\widehat{X}_0 \in \mathcal{BR}^{[-\epsilon,\epsilon]})] \\
                                                 &=\lim_{\epsilon \rightarrow 0}\sqrt{\frac{\pi}{2}}\frac{1}{\epsilon} \mathbb{E}^{\bf W}[f(w) 1(w_1 \in [-\epsilon,\epsilon])] \notag\\
                                                 &= \lim_{\epsilon \rightarrow 0} \mathbb{E}^{\bf W}[f(w)|w_1 \in [-\epsilon,\epsilon]] \notag\\
                                                 &\label{approx}= \mathbb{P}^{\bf W^0}f,
\end{align}
where $\mathbb{P}^{\bf W}$ is Wiener measure on $\mathcal{C}_0[0,1]$. The equality \eqref{sta} is due to stationarity of the moving-window process $\widehat{X}$, and the equality \eqref{approx} follows from the weak convergence to Brownian bridge of Brownian motion, see e.g. Billingsley \cite[Section $11$]{Bill}.  
\end{proof}
\begin{remark}
The measure $P_{\Gamma} \circ \Pi_0^{-1}$ defined in Proposition \ref{Revuz} is closely related to the notion of {\em Revuz measure} of Markov additive functionals. Note that for $s\in \mathbb{R}$ and $t \geq 0$, $\Gamma[s,s+t]=\Gamma[0,t] \circ \theta_s$, i.e. $\Gamma$ induces a continuous additive functional of the moving-window process $\widehat{X}$.\footnote{This generalizes the definition of continuous additive functionals of one-sided Markov processes, see e.g. the book of Sharpe \cite[Chapter IV]{Sharpebook} and the survey paper of Getoor \cite{Getoors} for background.} Since $(\widehat{X}_t;t \in \mathbb{R})$ is stationary with respect to $\mathbb{P}^{\bf W}$,
$$\mathbb{P}_{\Gamma} \circ \Pi_0^{-1}f:=\mathbb{E}\int_0^1f(\widehat{X}_t)\Gamma(dt)\quad \mbox{for}~f:\mathcal{C}[0,1] \rightarrow \mathbb{R}~\mbox{bounded measurable},$$
can be viewed as Revuz measure of $\Gamma$ in the two-sided setting. For further discussions on Revuz measure of additive functionals, we refer readers to Revuz \cite{Revuz}, Fukushima \cite{Fukurevuz}, and Fitzsimmons and Getoor \cite{FGrevuz} among others.
\end{remark}
\subsection{Brownian bridge in two-sided Brownian motion}
\label{32}
In this paragraph, we show that there exists a random time $\widehat{T} \in \mathbb{R}$ such that $(\widehat{B}_{\widehat{T}+u}-\widehat{B}_{\widehat{T}}; 0 \leq u \leq 1)$ has the same distribution as Brownian bridge $b^0$. In terms of the moving-window process $\widehat{X}$, we show that
\begin{proposition}
\label{tsc}
There exists a random time $\widehat{T} \in \mathbb{R}$ such that $\widehat{X}_{\widehat{T}}$ has the same distribution as $(b^0_u; 0 \leq u \leq 1)$
\end{proposition}

As mentioned in the introduction, our proof relies on a recent result of Last and Thorisson \cite{LT2014}, which establishes a dual relation between stationary random measures and mass-stationary ones in the Euclidian space. We refer readers to Last and Thorisson \cite{LT2009,LT2011} for the notion of mass-stationarity, which is an analog to point-stationarity of random point processes. 

Before proceeding further, we need the following notations. Recall that $(Z_t; t \in \mathbb{R})$ is a path-measurable process with a state space $(E,\mathcal{E})$ such that $(E^{\mathbb{R}}, \mathcal{E}^{\mathbb{R}})$ is time-invariant, and $\xi$ is a random $\sigma-$finite measure on $\mathbb{R}$. Let $N_{\xi}$ be a simple point process on $\mathbb{Z}$ such that 
$$i \in N_{\xi} \Longleftrightarrow \xi(i+[0,1))>0.$$
Next we associate each $j \in \mathbb{Z}$ to the point of $N_{\xi}$ that is closest to $j$, choosing the smaller one when there are two such points. Then we obtain a countable number of sets, each of which contains exactly one point of $N_{\xi}$. Let $D$ be the set that contains $0$, and $S$ be the vector from $N_{\xi}-$point in the set $D$ to $0$. 
\begin{figure}[h]
\begin{center}
\includegraphics[width=0.8 \textwidth]{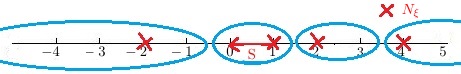}
\end{center}
\caption{Decomposition of $\mathbb{Z}$ induced by the simple point process $N_{\xi}$.}
\end{figure}

The next result is read from Last and Thorisson \cite[Theorem $2$]{LT2014}:
\begin{theorem} \cite{LT2014}
\label{thmLT}
Assume that $(1).$ the pair $(Z,\xi)$ is stationary, $(2).$ $\mathbb{E}\xi[0,1)<\infty$ and $(3).$ $\conv \supp \xi = \mathbb{R}$ a.s. where $\conv \supp \xi$ is the convex hull of the support of $\xi$.  Define
$$Z^0:=\theta_{T-S}Z,$$
where the conditional distribution of $T \in [0,1)$ given $(Z,\xi)$ is $$\theta_{-S}\xi(\cdot|[0,1)):=\frac{\theta_{-S}\xi(\cdot \cap [0,1))}{\theta_{-S}\xi([0,1))}.$$ Define
$$d\mathbb{P}^0:=\frac{\theta_{-S}\xi[0,1)}{\#D \cdot \mathbb{E}\xi[0,1)}d\mathbb{P},$$
where $\#D$ is the cardinality of the set $D$. Then $Z^0$ under $\mathbb{P}^0$ is the {\em Palm probability measure} of $\xi$.
\end{theorem}

Thorisson \cite{Th1995,Th1999} presented a duality between stationary point processes and point-stationary ones in the Euclidian space. In particular, the stationary point process and its (modified) Palm version are the same with some random time-shift. Thus Theorem \ref{thmLT} is regarded as a generalization of those results in the random diffuse measure setting.

Now we apply Theorem \ref{thmLT} to $Z:=\widehat{X}$ the moving-window process as in \eqref{tsmw}, and $\xi:=\Gamma$ the local times as in \eqref{localtime}. It is straightforward that all assumptions in Theorem \ref{thmLT} are satisfied. This leads to:
\begin{corollary}
\label{absex}
There exists a random time $T \in \mathbb{R}$ such that the {\em Palm probability measure} of $\Gamma$, i.e. $\mathbb{P}_{\Gamma}/\mathbb{P}_{\Gamma}1$ is absolutely continuous with respect to the distribution of $\theta_{T}\widehat{X}$.
\end{corollary}

By Proposition \ref{Revuz}, the $0$-marginal of the Palm probability measure of $\Gamma$ is Brownian bridge. If we can show that the Palm probability measure $\mathbb{P}_{\Gamma}/\mathbb{P}_{\Gamma}1$ is achieved by $\theta_{\widehat{T}}\widehat{X}$ for a random time $\widehat{T} \in \mathbb{R}$, then Proposition \ref{tsc} follows as a consequence. To this end, we state a general result, the proof of which is deferred.
\begin{theorem}
\label{generalerg}
Let $(Z_t; t \in \mathbb{R})$ be a path-measurable process with a state space $(E,\mathcal{E})$. Assume that  
\begin{enumerate}
\item
$Z$ is ergodic under the time-shift group $(\theta_t; t \in \mathbb{R})$, that is
$$\mathbb{P}(Z \in H)=0~\mbox{or}~1 \quad \mbox{for all}~H \in \mathcal{I}:=\{H' \in \mathcal{E}^{\mathbb{R}}; \theta_t H'=H'~\mbox{for all}~t \in \mathbb{R}\};$$
\item
$\mu$ is a probability measure on $(E^{\mathbb{R}},\mathcal{E}^{\mathbb{R}})$ absolutely continuous with respect to the distribution of $\theta_T Z$ for a random time $T \in \mathbb{R}$.
\end{enumerate}
 Then there exists a random time $\widehat{T} \in \mathbb{R}$ such that $\theta_{\widehat{T}}Z$ is distributed as $\mu$. 
\end{theorem}

\begin{proof}[Proof of Proposition \ref{tsc}] We apply Theorem \ref{generalerg} to $Z:=\widehat{X}$ the moving-window process and $\mu:=\mathbb{P}_{\Gamma}/\mathbb{P}_{\Gamma}1$ the Palm probability measure of $\Gamma$. Observe that the invariant $\sigma$-field $\mathcal{I} \subset \cap_{n \in \mathbb{N}} \theta_n^{-1}\mathcal{E}^{\mathbb{R}}$, the tail $\sigma$-field of $(\widehat{X}_n; n \in \mathbb{N})$ which are i.i.d. copies of Brownian motion on $[0,1]$. By Kolmogorov's zero-one law, $\mathcal{I}$ is trivial under the distribution of $\widehat{X}$ and the assumption $(1)$ is satisfied. The assumption $(2)$ follows from Corollary \ref{absex}. Combining Theorem \ref{generalerg} and Proposition \ref{Revuz}, we obtain the desired result.  
\end{proof}
\begin{remark}
\label{smixing}
 In ergodic theory, the process $Z$ is said to be $\theta$-{\em mixing} if 
$$\sup\{\mathbb{P}(Z \in A \cap B)-\mathbb{P}(Z \in A) \mathbb{P}(Z \in B); t \in \mathbb{R}, A \in \mathcal{F}_t, B \in \mathcal{F}^{t+s}\} \rightarrow 0 \quad \mbox{as}~s \rightarrow \infty,$$
where $\mathcal{F}_t:=\sigma(Z_u; u \leq t)$ and $\mathcal{F}^{t+s}:=\sigma(Z_u; u \geq t+s)$. See e.g. Bradley \cite{BR} for a survey on strong mixing conditions. It is quite straightforward that the moving-window process $\widehat{X}$ is $\theta$-mixing, since $\widehat{X}_{t+l}$ and $\widehat{X}_t$ are independent for all $t \geq 0$ and $l \geq 1$. Consequently, $\widehat{X}$ is ergodic under time-shift $(\theta_t; t \in \mathbb{R})$. In Section \ref{33}, this notion of $\theta$-mixing plays an important role in one-sided embedding out of two-sided processes.
\end{remark}
In the rest of this part, we aim to prove Theorem \ref{generalerg}. We need the following result of Thorisson \cite{Th2}, which provides a necessary and sufficient condition for two continuous-time processes being transformed from one to the other by a random time-shift.
\begin{theorem} \cite{Th2}
\label{thts}
Let $(Z_t; t \in \mathbb{R})$ and $(Z'_t; t \in \mathbb{R})$ be two path-measurable processes with a state space $(E,\mathcal{E})$. Then there exists a random time $\widehat{T} \in \mathbb{R}$ such that $\theta_{\widehat{T}}Z $ has the same distribution as $Z'$ if and only if the distributions of $Z$ and $Z'$ agree on the invariant $\sigma-$field $\mathcal{I}$.
\end{theorem}
\begin{proof}[Proof of Theorem \ref{generalerg}] We apply Theorem \ref{thts} to $Z'$ distributed as $\mu$. Since $(\theta_t; t \in \mathbb{R})$ is ergodic under the distribution of $Z$, 
$$\mathbb{P}(Z\in H)=0~\mbox{or}~1 \quad \mbox{for all}~H \in \mathcal{I}.$$
If $\mathbb{P}(Z \in H)=0$ for $H \in \mathcal{I}$, then $\mathbb{P}(\theta_T Z \in H)=\mathbb{P}(Z \in \theta_{-T}H)=\mathbb{P}(Z\in H)=0$. By the absolute continuity between the distribution $\mu$ and that of $\theta_TZ$, we have $\mu(H)=0$. By applying the same argument to the complement of $H$, $\mathbb{P}(Z \in H)=1$ for $H \in \mathcal{I}$ implies $\mu(H)=1$. Thus, the probability distribution $\mu$ and that of $Z$ agree on the invariant $\sigma-$field $\mathcal{I}$. Theorem \ref{thts} permits to conclude. 
\end{proof}
\subsection{From two-sided embedding to one-sided embedding}
\label{33}
We explain here how to achieve a certain distribution on $\mathcal{C}_0[0,1]$ in Brownian motion by a random spacetime shift, once this has been done in two-sided Brownian motion. We aim to prove that:
\begin{proposition}
\label{tto}
Assume that $\mu$ is a probability measure on $(\mathcal{C}_0[0,1],\mathcal{B})$. If $\widehat{X}_{\widehat{T}}$ is distributed as $\mu$ for some random time $\widehat{T} \in \mathbb{R}$, then there exists a random time $T \geq 0$ such that $\widehat{X}_T$ is distributed as $\mu$.
\end{proposition}

It is not hard to see that Theorem \ref{mainbis} follows readily from Corollary \ref{tsc} and Proposition \ref{tto}. In the sequel, we assume path-measurability for any continuous-time processes that are involved. Let $\mathcal{L}(\mathcal{X})$ be the distribution of any random element $\mathcal{X}$. To prove Proposition \ref{tto}, we begin with a general statement.
\begin{proposition}
\label{lemtto}
Let $(Z_t; t \in \mathbb{R})$ be a stationary process and $\theta$-mixing as in Remark \ref{smixing}. Assume that $Z_{\widehat{T}}$ is distributed as $\mu$ for some random time $\widehat{T} \in \mathbb{R}$. Given $\epsilon>0$ and $N \in \mathbb{N}$, there exist random times $0 \leq T_1< \cdots <T_N$ on some event $E_N$ of probability larger than $1-\epsilon$ such that 
$$||\mathcal{L}(Z_{T_1}, \cdots, Z_{T_N}|E_N)- \mu^{\otimes N}||_{TV} \leq \epsilon,$$
where $||\cdot||_{TV}$ is the total variation norm of a measure.
\end{proposition}

Before proceeding the proof, we need the following lemma known as Blackwell-Dubins' merging of opinions \cite{BD}. In that paper, they only proved the result for discrete chains. But the argument can be easily adapted to the continuous setting. We rewrite Blackwell-Dubins' theorem for our own purpose, and leave full details to careful readers.
\begin{lemma} \cite{BD} \label{BDmo}
Let $(Z_t; t \in \mathbb{R})$ and $(Z'_t; t \in \mathbb{R})$ be two path-measurable processes with a state space $(E,\mathcal{E})$. If the distribution of $Z'$ is absolutely continuous with respect to $Z$, then
$$||\mathcal{L}(Z'_{t+s};s \geq 0|\mathcal{F}'_t)-\mathcal{L}(Z_{t+s};s \geq 0|\mathcal{F}_t)||_{TV} \rightarrow 0 \quad \mbox{as}~t \rightarrow \infty,$$
where $\mathcal{F}_t:=\sigma(Z_u; u \leq t)$ and $\mathcal{F}'_t:=\sigma(Z'_u; u \leq t)$.
\end{lemma}
\begin{proof}[Proof of Proposition \ref{lemtto}] We proceed by induction over $N \in \mathbb{N}$. By stationarity of $Z$, for each $s \in \mathbb{R}$, $\theta_sZ$ has the same distribution as $Z$. Let $\widehat{T}_{\theta_s}$ be the random time constructed from $\theta_sZ$ just as $\widehat{T}$ is constructed from $Z$. Therefore, $(\theta_sZ)_{\widehat{T}_{\theta_s}}=Z_{\widehat{T}_{\theta_s}+s}$ is distributed as $\mu$. Let $t \in \mathbb{R}$ be the $\frac{\epsilon}{2}$-quantile of $\widehat{T}$, that is $\mathbb{P}(\widehat{T} \geq t) \geq 1-\frac{\epsilon}{2}$. Define $T_1:=\widehat{T}_{\theta_{-t}}-t$. Observe that for $A \in \mathcal{E}$,
$$\mathbb{P}(Z_{T_1} \in A)-\frac{\epsilon}{2} \leq \mathbb{P}(Z_{T_1} \in A~\mbox{and}~T_1 \geq 0) \leq \mathbb{P}(Z_{T_1} \in A).$$
As a consequence, on the event $E_1:=\{T_1 \geq 0\}$ of probability larger than $1-\frac{\epsilon}{2}>1-\epsilon$,  
$$||\mathcal{L}(Z_{T_1}|E_1) -\mu||_{TV} \leq \frac{\epsilon/2}{1-\epsilon/2}<\epsilon.$$
Suppose that there exist $0 \leq T_1< \cdots <T_N$ on some event $E_N$ of probability larger than $1-\frac{\epsilon}{4}$ such that 
\begin{equation}
\label{fm1}
||\mathcal{L}(Z_{T_1}, \cdots, Z_{T_N}|E_N)-\mu^{\otimes N}||_{TV} \leq \frac{\epsilon}{4}.
\end{equation}
Note that the distribution of the conditioned moving-window process $(Z_s; s \in  \mathbb{R}|E_N)$ is absolutely continuous with respect to that of the original one $(Z_s; s \in \mathbb{R})$. By Lemma \ref{BDmo}, 
\begin{equation}
\label{59}
||\mathcal{L}(Z_{t+s}; s \geq 0|E_N, \mathcal{F}_t)-\mathcal{L}(Z_{t+s};s \geq 0|\mathcal{F}_t)||_{TV} \rightarrow 0 \quad \mbox{as}~t \rightarrow \infty.
\end{equation}
By triangle inequality, we have for $t',t'' \geq 0$,
\begin{align}
\label{5122}
&||\mathcal{L}(Z_{t'+t''+s};s \geq 0|E_N) - \mathcal{L}(Z_s; s \geq 0)||_{TV} \notag\\
&\quad \quad  \leq||\mathcal{L}(Z_{t'+t''+s};s \geq 0|E_N) - \mathcal{L}(Z_{t'+t''+s};s \geq 0|E_N, \mathcal{F}_{t'})||_{TV} \notag\\
&\quad \quad \quad  \quad \quad+ ||\mathcal{L}(Z_{t'+t''+s};s \geq 0|E_N, \mathcal{F}_{t'}) - \mathcal{L}(Z_{t'+t''+s};s \geq 0|\mathcal{F}_{t'})||_{TV} \notag\\
&\quad \quad \quad  \quad \quad \quad \quad \quad \quad \quad \quad \quad \quad+||\mathcal{L}(Z_{t'+t''+s};s \geq 0|\mathcal{F}_{t'}) - \mathcal{L}(Z_{s};s \geq 0)||_{TV}.
\end{align}
By $\theta$-mixing property, the first and the third term of \eqref{5122} goes to $0$ as $t'' \rightarrow \infty$. By \eqref{59}, the second term of \eqref{5122} goes to $0$ as $t' \rightarrow \infty$. Therefore,
\begin{equation}
\lim_{t \rightarrow \infty} ||\mathcal{L}(Z_{t+s};s \geq 0|E_N) - \mathcal{L}(Z_s; s \geq 0)||_{TV} =0 \quad \mbox{as}~t \rightarrow \infty.
\end{equation}
Pick $t_N \geq 0$ such that $\mathbb{P}(T_N \geq t_N) \leq \frac{\epsilon}{8}$ and
$$||\mathcal{L}(Z_{t_N+s}; s \geq 0|E_N)-\mathcal{L}(Z_s; s \geq 0)||_{TV} \leq \frac{\epsilon}{8}.$$
By a similar argument as in the case of $N=1$, there exists a random time $T_{N+1} \in \mathbb{R}$ such that $\mathbb{P}(T_{N+1} \geq t_N) \geq 1-\frac{\epsilon}{8}$ and 
\begin{equation}
\label{fm2}
||\mathcal{L}(Z_{T_{N+1}}|E_N \cap \{T_{N+1} \geq t_N\})-\mu||_{TV} \leq \frac{\epsilon}{8}+\frac{\epsilon}{8}=\frac{\epsilon}{4}.
\end{equation}
Let $E_{N+1}:=E_N \cap \{T_{N+1} >T_N\}$. Since $$\mathbb{P}(T_{N+1}>T_N) \geq \mathbb{P}(T_{N+1} \geq t_N)-\mathbb{P}(T_N>t_N) \geq 1-\frac{\epsilon}{4},$$ 
we get $\mathbb{P}(E_{N+1}) \geq 1-\frac{\epsilon}{2}>1-\epsilon$. By \eqref{fm1} and \eqref{fm2},
$$||\mathcal{L}(Z_{T_1}, \cdots, Z_{T_N}|E_{N+1})-\mu^{\otimes N}||_{TV} \leq \frac{\epsilon}{2} \quad \mbox{and} \quad ||\mathcal{L}(Z_{T_{N+1}}|E_{N+1})-\mu||_{TV} \leq \frac{\epsilon}{2}.$$
We obtain immediately that $||\mathcal{L}(Z_{T_1}, \cdots, Z_{T_{N+1}}|E_{N+1})-\mu^{\otimes N+1}||_{TV} \leq \frac{\epsilon}{2}+\frac{\epsilon}{2}=\epsilon$.  
\end{proof}

By applying Lemma \ref{BDmo} in the first step, we deduce from Proposition \ref{lemtto} that:
\begin{corollary}
\label{impc}
Let $(Z_t; t \in \mathbb{R})$ be a stationary process and $\theta$-mixing as in Remark \ref{smixing}. Assume that $Z_{\widehat{T}}$ is distributed as $\mu$ for some random time $\widehat{T} \in \mathbb{R}$. Given $\epsilon>0$, $N \in \mathbb{N}$ and $E_0$ an event of positive probability, there exist random times $0 \leq T_1 < \cdots <T_N$ on some event $E_N$ of probability larger than $1-\epsilon$ such that
$$||\mathcal{L}(Z_{T_1}, \cdots, Z_{T_N}|E_0, E_N)- \mu^{\otimes N}||_{TV} \leq \epsilon.$$
\end{corollary}

Now let us recall some elements of {\em von Neumann's acceptance-rejection algorithm} \cite{vN}. Assume that $\mu$ and $\nu$ are two probability measures such that the Radon-Nikodym derivative $f:=\frac{d\nu}{d\mu}$ is essentially bounded under $\mu$. Let $(Z_n)_{n \in \mathbb{N}} \sim \mu^{\otimes \mathbb{N}}$ be a sequence of i.i.d. random variables distributed as $\mu$. Then
$$
Z_T \sim \nu \quad \mbox{with}~T:=\inf\left\{i \in \mathbb{N}; U_i  \leq \frac{f(Z_i)}{\ess \sup f}\right\},
$$
where $(U_n)_{n \in \mathbb{N}}$ is a sequence of i.i.d. uniform-$[0,1]$ random variables independent of $(Z_n)_{n \in \mathbb{N}}$. It is well-known that the total variation between the $N^{th}$ updated distribution and the target one is of geometric decay, i.e.
$$||\mathcal{L}(Z_{T \wedge N})-\nu||_{TV} \leq 2 \left(1-\frac{1}{\ess \sup f}\right)^N.$$
If the sample size $N$ is large enough, a good portion of the target distribution $\nu$ can be sampled from $(Z_1,\cdots, Z_N) \sim \mu^{\otimes N}$ \`{a} la von Neumann. The following lemma is a slight extension of the above result to the quasi-i.i.d. case. The proof is quite standard, and thus is omitted.
\begin{lemma}
\label{quasiiid}
Assume that $||\mathcal{L}(Z_1 \cdots Z_N)-\mu^{\otimes N}||_{TV} \leq \epsilon$ for some $\epsilon >0$ and $N \in \mathbb{N}$.Then
$$||\mathcal{L}(Z_{T_N})-\nu||_{TV} \leq \epsilon +2 \left(1-\frac{1}{\ess \sup f}\right)^N,$$
where $T_N:=\inf\{i \leq N; U_i   \leq f(Z_i)/\ess \sup f\} \wedge N$.
\end{lemma}

\begin{proof}[Proof of Proposition \ref{tto}] We use the same notation as in the proof of Proposition \ref{lemtto}. Let $t \in \mathbb{R}$ be the $\frac{1}{2}-$quantile of $\widehat{T}$, and define $T_1:=\widehat{T}_{\theta_{-t}}$. By taking $T_1 \geq 0$ as the stopping rule, we obtain a $\frac{1}{2}$ portion of $\mu$. The idea now is to get the remaining $\frac{1}{2}$ portion of $\mu$ by filling-type argument. Note that the target distribution $\mathcal{L}(\widehat{X}_{T_1}|T_1<0)$ is absolutely continuous with respect to $\mu$ with the Radon-Nikodym density
$$f_1:=\frac{d \mathcal{L}(\widehat{X}_{T_1}|T_1<0)}{d \mu},$$
which is bounded by $2$. As indicated in Remark \ref{smixing}, the moving-window process $\widehat{X}$ is stationary and $\theta$-{\em mixing}. We apply Corollary \ref{impc} to $(\widehat{X}_t; t \in \mathbb{R}|T_1<0)$: for any fixed $N \in \mathbb{N}$, there exist random times $0 \leq T_1< \cdots <T_N$ on some event $E_N$ of probability larger than $\frac{3}{4}$ such that
$$||\mathcal{L}(\widehat{X}_{T_1}, \cdots, \widehat{X}_{T_N}|T_1<0, E_N)-\mu^{\otimes N}||_{TV} \leq \frac{1}{4}.$$
By Lemma \ref{quasiiid}, there is a random integer $M  \leq N$ such that 
$$||\mathcal{L}(\widehat{X}_{T_M}|T_1<0, E_N)-\mathcal{L}(\widehat{X}_{T_1}|T_1<0)||_{TV} \leq \frac{1}{4}+2 \left(1-\frac{1}{\ess \sup f_1}\right)^N .$$
By taking $N \in \mathbb{N}$ such that $(1-1/\ess \sup f_1)^N \leq \frac{1}{8}$, we retrieve a $\frac{1}{2}$ portion of the targeted $\mathcal{L}(\widehat{X}_{T_1}|T_1<0)$. Restricted to the sub-probability space $\{T_1<0\}$, we obtain a remaining $\frac{1}{4}$ portion of $\mu$. Repeat the algorithm, and we finally achieve the desired ($\frac{1}{2}+\frac{1}{4}+\frac{1}{8}+\cdots$) distribution $\mu$.  
\end{proof}
\subsection{An explicit embedding of Brownian bridge into Brownian motion}
\label{34}
In this subsection, we present a constructive proof of Theorem \ref{mainbis} due to Hermann Thorisson. Recall that $\widehat{X}$ is the moving-window process associated to a two-sided Brownian motion.
\begin{theorem} \cite{HTP}
\label{HHM}
Let $\Gamma$ be the random $\sigma$-finite measure defined as in \eqref{localtime}. Define
\begin{equation}
\label{finally}
T:=\inf\{t >0; \Gamma[0,t]=t\}.
\end{equation}
Then $T<\infty$ almost surely, and $\widehat{X}_T$ has the same distribution as standard Brownian bridge $(b^0_u; 0 \leq u \leq 1)$.
\end{theorem}

To proceed further, we need the following notions regarding transports of random measures, initiated by Holroyd and Peres \cite{HP2005}, and Last and Thorisson \cite{LT2009}.
\begin{definition}
\cite{HP2005,LT2009}
Let $(\Omega,\mathcal{F}, \mathbb{P})$ be a generic probability space, equipped with a flow $(\theta_t; t \in \mathbb{R})$.
\begin{enumerate}
\item
A measurable mapping $\tau: \Omega \times \mathbb{R} \rightarrow \mathbb{R}$ is called an allocation rule if
$$\tau(\theta_t w, s-t)=\tau(w,s)-t \quad \mbox{for}~s,t \in \mathbb{R} \quad \mathbb{P}~a.s.$$
\item
An allocation rule $\tau$ is said to balance two random measures $\xi$ and $\eta$ if
$$\int_{\mathbb{R}}1(\tau(s) \in \cdot) \xi(ds)=\eta \quad \mathbb{P}~a.s.$$
\end{enumerate}
\end{definition}

Triggered by the work of Liggett \cite{Liggett}, and Holroyd and Liggett \cite{HL} on transporting counting measures on $\mathbb{Z}^d$ to the Bernoulli random measure, allocation rules of counting measures on $\mathbb{Z}^d$ to an ergodic point process have received much attention, see e.g. Holroyd and Peres \cite{HP2005}, Hoffman et al. \cite{HHP}, Chatterjee et al. \cite{CPPR}, and Last and Thorisson \cite{LT2009} among others. The following result of Last et al. \cite[Theorem $5.1$]{LMT} gives a balancing allocation rule for diffuse random measures on the line.
\begin{theorem} \cite{LMT}
\label{ppl}
Let $\xi$ and $\eta$ be invariant orthogonal diffuse random measures on $\mathbb{R}$ with finite intensities. Assume that
\begin{equation}
\label{cddd}
\mathbb{E}[\xi[0,1]| \mathcal{I}] =\mathbb{E}[\eta[0,1]|\mathcal{I}] \quad a.s.,
\end{equation}
where $\mathcal{I}$ is the invariant $\sigma$-field. Then the mapping 
$$\tau(s):=\inf\{t>s; \xi[s,t]=\eta[s,t]\} \quad \mbox{for all}~ s \in \mathbb{R}$$
is an allocation rule balancing $\xi$ and $\eta$.
\end{theorem}
\begin{corollary}
Let $\Gamma$ be the random $\sigma$-finite measure defined as in \eqref{localtime}. Define
\begin{equation}
T(s):=\inf \{t>s; \Gamma[s,t]=t-s\} \quad \mbox{for all}~s \in \mathbb{R}.
\end{equation}
Then $(T(s); s \in \mathbb{R})$ is an allocation rule balancing the Lebesgue measure $\mathcal{L}^1$ on $\mathbb{R}$ and $\Gamma$.
\end{corollary}
\begin{proof} We want to apply Theorem \ref{ppl} to $\xi:=\mathcal{L}^1$ and $\eta:=\Gamma$. We need to check the conditions. First it is obvious that $\mathcal{L}^1$ and $\Gamma$ are invariant diffuse measures on the real line. Note that the measure $\Gamma$ is supported on the set $\{t \in \mathbb{R}; \widehat{X}_t \in \mathcal{BR}^0\}$ almost surely. The distribution of $\{t \in \mathbb{R}; \widehat{X}_t \in \mathcal{BR}^0\}$ is the same as that of $\{t \in \mathbb{R}; S_t=0\}$, which has null Lebesgue measure almost surely. Therefore, the measures $\mathcal{L}^1$ and $\Gamma$ are orthogonal. In the proof of Proposition \ref{tsc}, we know that $\mathcal{I}$ is trivial under the distribution of $\widehat{X}$. In addition, $\mathbb{E}\Gamma[0,1]=1$ by the computation as in Proposition \ref{Revuz}. Thus, we have the condition \eqref{cddd} in the case of $\xi:=\mathcal{L}^1$ and $\eta:=\Gamma$. 
\end{proof}

In terms of Palm measures, Last and Thorisson \cite[Theorem $4.1$]{LT2009} gave a necessary and sufficient condition for an allocation rule to balance two random measures. See also Last et al. \cite[Theorem $2.1$]{LMT}.
\begin{theorem} \cite{LT2009}
\label{lstt}
Consider two random measures $\xi$ and $\eta$ on $\mathbb{R}$ and an allocation rule $\tau$. Then $\tau$ balances $\xi$ and $\eta$ if and only if
$$\mathbb{P}_{\xi}(\theta_{\tau(0)}w \in \cdot)=\mathbb{P}_{\eta},$$
where $\mathbb{P}_{\xi}$ (resp. $\mathbb{P}_{\eta}$) is the Palm measure of $\xi$ (resp. $\eta$).
\end{theorem}
\begin{proof}[Proof of Theorem \ref{HHM}] Applying Theorem \ref{lstt} to $\xi:=\mathcal{L}^1$, $\eta:=\Gamma$ and $\tau:=T$, we get:
$$\mathbb{P}_{\mathcal{L}^1}(\theta_T \widehat{X} \in \cdot)=\mathbb{P}_{\Gamma}.$$
where $T$ is defined by \eqref{finally}. Note that for a stationary process, the Palm version of the Lebesgue measure is the stationary process itself. In particular, $\mathbb{P}_{\mathcal{L}^1}$ is the distribution of the moving-window process $\widehat{X}$. By Proposition \ref{Revuz}, the $0$-marginal of $\mathbb{P}_{\Gamma}$ is standard Brownian bridge. Theorem \ref{HHM} follows readily from these facts.
\end{proof}

Finally, let us derive a simple consequence of Theorem \ref{HHM}. Let $(b^{ps}_u; 0 \leq u \leq 1)$ be the {\em pseudo Brownian bridge} defined by
$$b^{ps}_u:=\frac{B_{u\tau_1}}{\sqrt{\tau_1}} \quad \mbox{for all}~0 \leq u \leq 1,$$
where $\tau_1:=\inf\{t \geq 0; L_t>1\}$ is the inverse local times of Brownian motion. Biane et al. \cite{BLY} proved that the distribution of the pseudo Brownian bridge is mutually absolutely continuous relative to that of standard Brownian bridge. That is, for $f: \mathcal{C}_{0}[0,1] \rightarrow \mathbb{R}$ a bounded measurable function,
$$\mathbb{E}[f(b^{ps}_u; 0 \leq u \leq 1)]=\mathbb{E}\left[\sqrt{\frac{2}{\pi}}L_1^{br}f(b^0_u; 0 \leq u \leq 1)\right],$$
where $L^{br}_1$ is the local time of Brownian bridge at level $0$ up to time $1$.

From Theorem \ref{HHM}, we are able to find a sequence of i.i.d. Brownian bridges by iteration of the construction \eqref{finally}. According to Pitman and Tang \cite[Section $3.5$]{PTpattern}, we can apply Rost's filling scheme \cite{CO,Rost1} to sample distributions which are absolutely continuous relative to that of Brownian bridge. In particular,
\begin{corollary}
There exists a random time $T \geq 0$ such that $(B_{T+u}-B_T; 0 \leq u \leq 1)$ has the same distribution as $(b^{ps}_u; 0 \leq u \leq 1)$.
\end{corollary}


\begin{thebibliography}{99}

\bibitem{Abr}
Julia Abrahams.
\newblock A survey of recent progress on level-crossing problems for random
  processes.
\newblock In {\em Communications and Networks}, pages 6--25. Springer, 1986.

\bibitem{AT} David~J. Aldous and Hermann Thorisson. \newblock Shift-coupling. \newblock {\em Stochastic Process. Appl.}, 44(1):1--14, 1993. \MR{1198659}

\bibitem{AW} Jean-Marc Aza{\"{\i}}s and Mario Wschebor. \newblock {\em Level sets and extrema of random processes and fields}. \newblock John Wiley \& Sons, Inc., Hoboken, NJ, 2009. \MR{2478201}

\bibitem{Barlow} M.~T. Barlow. \newblock Study of a filtration expanded to include an honest time. \newblock {\em Z. Wahrsch. Verw. Gebiete}, 44(4):307--323, 1978. \MR{0509204}

\bibitem{Berman} Simeon~M. Berman. \newblock Local times and sample function properties of stationary {G}aussian  processes. \newblock {\em Trans. Amer. Math. Soc.}, 137:277--299, 1969. \MR{0239652}

\bibitem{BertoinLeJan} J.~Bertoin and Y.~Le~Jan. \newblock Representation of measures by balayage from a regular recurrent  point. \newblock {\em Ann. Probab.}, 20(1):538--548, 1992. \MR{1143434}

\bibitem{BCP} Jean Bertoin, Lo{\"{\i}}c Chaumont, and Jim Pitman. \newblock Path transformations of first passage bridges. \newblock {\em Electron. Comm. Probab.}, 8:155--166 (electronic), 2003. \MR{2042754}

\bibitem{BLY} Ph. Biane, J.-F. Le~Gall, and M.~Yor. \newblock Un processus qui ressemble au pont brownien. \newblock In {\em S\'eminaire de {P}robabilit\'es, {XXI}}, volume 1247 of {\em  Lecture Notes in Math.}, pages 270--275. Springer, Berlin, 1987. \MR{0941990}

\bibitem{BY} Ph. Biane and M.~Yor. \newblock Quelques pr\'ecisions sur le m\'eandre brownien. \newblock {\em Bull. Sci. Math. (2)}, 112(1):101--109, 1988. \MR{0942801}

\bibitem{Bill} Patrick Billingsley. \newblock {\em Convergence of probability measures}. \newblock John Wiley \& Sons, Inc., New York-London-Sydney, 1968. \MR{0233396}

\bibitem{Bismut} Jean-Michel Bismut. \newblock Last exit decompositions and regularity at the boundary of transition  probabilities. \newblock {\em Z. Wahrsch. Verw. Gebiete}, 69(1):65--98, 1985. \MR{0775853}

\bibitem{BD} David Blackwell and Lester Dubins. \newblock Merging of opinions with increasing information. \newblock {\em Ann. Math. Statist.}, 33:882--886, 1962. \MR{0149577}

\bibitem{BL} Ian~F. Blake and William~C. Lindsey. \newblock Level-crossing problems for random processes. \newblock {\em IEEE Trans. Information Theory}, IT-19:295--315, 1973. \MR{0370729}

\bibitem{Blumenthal} R.~M. Blumenthal. \newblock An extended {M}arkov property. \newblock {\em Trans. Amer. Math. Soc.}, 85:52--72, 1957. \MR{0088102}

\bibitem{BR} Richard~C. Bradley. \newblock Basic properties of strong mixing conditions. {A} survey and some  open questions. \newblock {\em Probab. Surv.}, 2:107--144, 2005. \newblock Update of, and a supplement to, the 1986 original. \MR{2178042}

\bibitem{Burdzy} Krzysztof Burdzy and Michael Scheutzow. \newblock Forward {B}rownian motion. \newblock {\em Probab. Theory Related Fields}, 160(1-2):95--126, 2014. \MR{3256810}

\bibitem{CO} R.~V. Chacon and D.~S. Ornstein. \newblock A general ergodic theorem. \newblock {\em Illinois J. Math.}, 4:153--160, 1960. \MR{0110954}

\bibitem{CPPR} Sourav Chatterjee, Ron Peled, Yuval Peres, and Dan Romik. \newblock Phase transitions in gravitational allocation. \newblock {\em Geom. Funct. Anal.}, 20(4):870--917, 2010. \MR{2729280}

\bibitem{CL} Harald Cram{\'e}r and M.~R. Leadbetter. \newblock {\em Stationary and related stochastic processes. {S}ample function  properties and their applications}. \newblock John Wiley \& Sons, Inc., New York-London-Sydney, 1967. \MR{0217860}

\bibitem{Doob} J.~L. Doob. \newblock The {B}rownian movement and stochastic equations. \newblock {\em Ann. of Math. (2)}, 43:351--369, 1942. \MR{0006634}

\bibitem{DJ} E.~Dynkin and A.~Jushkevich. \newblock Strong {M}arkov processes. \newblock {\em Teor. Veroyatnost. i Primenen.}, 1:149--155, 1956. \MR{0088103}

\bibitem{FGrevuz} P.~J. Fitzsimmons and R.~K. Getoor. \newblock Smooth measures and continuous additive functionals of right {M}arkov  processes. \newblock In {\em It\^o's stochastic calculus and probability theory}, pages  31--49. Springer, Tokyo, 1996. \MR{1439516}

\bibitem{FPY} Pat Fitzsimmons, Jim Pitman, and Marc Yor. \newblock Markovian bridges: construction, {P}alm interpretation, and splicing. \newblock In {\em Seminar on {S}tochastic {P}rocesses, 1992 ({S}eattle, {WA},  1992)}, volume~33 of {\em Progr. Probab.}, pages 101--134. Birkh\"auser  Boston, Boston, MA, 1993. \MR{1278079}

\bibitem{Fukurevuz} Masatoshi Fukushima. \newblock On additive functionals admitting exceptional sets. \newblock {\em J. Math. Kyoto Univ.}, 19(2):191--202, 1979. \MR{0545703}

\bibitem{GH} Donald Geman and Joseph Horowitz. \newblock Occupation densities. \newblock {\em Ann. Probab.}, 8(1):1--67, 1980. \MR{0556414}

\bibitem{Gthesis}
Daniel~Sebastian Gentner.
\newblock {\em Palm Theory, Mass Transports and Ergodic Theory for
  Group-stationary Processes}.
\newblock KIT Scientific Publishing, 2011.

\bibitem{Getoors} R.~K. Getoor. \newblock Additive functionals and excessive functions. \newblock {\em Ann. Math. Statist.}, 36:409--422, 1965. \MR{0172335}

\bibitem{Getoorsplit} R.~K. Getoor. \newblock Splitting times and shift functionals. \newblock {\em Z. Wahrsch. Verw. Gebiete}, 47(1):69--81, 1979. \MR{0521533}

\bibitem{GetSh} R.~K. Getoor and M.~J. Sharpe. \newblock Last exit times and additive functionals. \newblock {\em Ann. Probability}, 1:550--569, 1973. \MR{0353468}

\bibitem{GetShtri} R.~K. Getoor and M.~J. Sharpe. \newblock Last exit decompositions and distributions. \newblock {\em Indiana Univ. Math. J.}, 23:377--404, 1973/74. \MR{0334335}

\bibitem{GetShbis} R.~K. Getoor and M.~J. Sharpe. \newblock The {M}arkov property at co-optional times. \newblock {\em Z. Wahrsch. Verw. Gebiete}, 48(2):201--211, 1979. \MR{0534845}

\bibitem{GSde} R.~K. Getoor and M.~J. Sharpe. \newblock Excursions of dual processes. \newblock {\em Adv. in Math.}, 45(3):259--309, 1982. \MR{0673804}

\bibitem{HPS}
Alan Hammond, G\'{a}bor Pete, and Oded Schramm.
\newblock Local time on the exceptional set of dynamical percolation, and the
  {I}ncipient {I}nfinite {C}luster.
\newblock 2012.
\newblock arXiv: 1208.3826.

\bibitem{HHP} Christopher Hoffman, Alexander~E. Holroyd, and Yuval Peres. \newblock A stable marriage of {P}oisson and {L}ebesgue. \newblock {\em Ann. Probab.}, 34(4):1241--1272, 2006. \MR{2257646}

\bibitem{HL} Alexander~E. Holroyd and Thomas~M. Liggett. \newblock How to find an extra head: optimal random shifts of {B}ernoulli and  {P}oisson random fields. \newblock {\em Ann. Probab.}, 29(4):1405--1425, 2001. \MR{1880225}

\bibitem{HP2005} Alexander~E. Holroyd and Yuval Peres. \newblock Extra heads and invariant allocations. \newblock {\em Ann. Probab.}, 33(1):31--52, 2005. \MR{2118858}

\bibitem{Hunt} G.~A. Hunt. \newblock Some theorems concerning {B}rownian motion. \newblock {\em Trans. Amer. Math. Soc.}, 81:294--319, 1956. \MR{0079377}

\bibitem{Itoex}
Kyosi It{\^o}.
\newblock Poisson point processes attached to markov processes.
\newblock In {\em Proc. 6th Berk. Symp. Math. Stat. Prob}, volume~3, pages
  225--240, 1971.

\bibitem{Jacob} Martin Jacobsen. \newblock Splitting times for {M}arkov processes and a generalised {M}arkov  property for diffusions. \newblock {\em Z. Wahrscheinlichkeitstheorie und Verw. Gebiete}, 30:27--43,  1974. \MR{0375477}

\bibitem{JeulinYor} T.~Jeulin and M.~Yor. \newblock Grossissement d'une filtration et semi-martingales: formules  explicites. \newblock In {\em S\'eminaire de {P}robabilit\'es, {XII} ({U}niv. {S}trasbourg,  {S}trasbourg, 1976/1977)}, volume 649 of {\em Lecture Notes in Math.}, pages  78--97. Springer, Berlin, 1978. \MR{0519998}

\bibitem{Jeulin} Thierry Jeulin. \newblock {\em Semi-martingales et grossissement d'une filtration}, volume 833  of {\em Lecture Notes in Mathematics}. \newblock Springer, Berlin, 1980. \MR{0604176}

\bibitem{Kallenberg} Olav Kallenberg. \newblock {\em Foundations of modern probability}. \newblock Probability and its Applications (New York). Springer-Verlag, New  York, second edition, 2002. \MR{1876169}

\bibitem{Kestenperco} Harry Kesten. \newblock The incipient infinite cluster in two-dimensional percolation. \newblock {\em Probab. Theory Related Fields}, 73(3):369--394, 1986. \MR{0859839}

\bibitem{Knightapprox1} F.~B. Knight. \newblock On the random walk and {B}rownian motion. \newblock {\em Trans. Amer. Math. Soc.}, 103:218--228, 1962. \MR{0139211}

\bibitem{Knightapprox2} F.~B. Knight. \newblock Random walks and a sojourn density process of {B}rownian motion. \newblock {\em Trans. Amer. Math. Soc.}, 109:56--86, 1963. \MR{0154337}

\bibitem{Kolm}
Andrey~N Kolmogorov.
\newblock Sulla determinazione empirica di una legge di distribuzione.
\newblock {\em Giornale dell’Istituto Italiano degli Attuari}, 4(1):83--91,
  1933.

\bibitem{Kratz} Marie~F. Kratz. \newblock Level crossings and other level functionals of stationary {G}aussian  processes. \newblock {\em Probab. Surv.}, 3:230--288, 2006. \MR{2264709}

\bibitem{LT2014}
G.~Last and H.~Thorisson.
\newblock Construction and characterization of stationary and mass-staionary
  random measures on $\mathbb{R}^d$.
\newblock 2014.
\newblock arXiv:1405.7566.

\bibitem{Last2008} G{\"u}nter Last. \newblock Modern random measures: {P}alm theory and related models. \newblock In {\em New perspectives in stochastic geometry}, pages 77--110.  Oxford Univ. Press, Oxford, 2010. \MR{2654676}

\bibitem{Last2010} G{\"u}nter Last. \newblock Stationary random measures on homogeneous spaces. \newblock {\em J. Theoret. Probab.}, 23(2):478--497, 2010. \MR{2644871}

\bibitem{LMT} G{\"u}nter Last, Peter M{\"o}rters, and Hermann Thorisson. \newblock Unbiased shifts of {B}rownian motion. \newblock {\em Ann. Probab.}, 42(2):431--463, 2014. \MR{3178463}

\bibitem{LT2009} G{\"u}nter Last and Hermann Thorisson. \newblock Invariant transports of stationary random measures and  mass-stationarity. \newblock {\em Ann. Probab.}, 37(2):790--813, 2009. \MR{2510024}

\bibitem{LT2011} G{\"u}nter Last and Hermann Thorisson. \newblock What is typical? \newblock {\em J. Appl. Probab.}, 48A(New frontiers in applied probability: a  Festschrift for Soren Asmussen):379--389, 2011. \MR{2865959}

\bibitem{Levy}
Paul L{\'e}vy.
\newblock Sur certains processus stochastiques homog{\`e}nes.
\newblock {\em Compositio mathematica}, 7:283--339, 1940.

\bibitem{Levybook} Paul L{\'e}vy. \newblock {\em Processus {S}tochastiques et {M}ouvement {B}rownien. {S}uivi  d'une note de {M}. {L}o\`eve}. \newblock Gauthier-Villars, Paris, 1948. \MR{0029120}

\bibitem{Liggett} Thomas~M. Liggett. \newblock Tagged particle distributions or how to choose a head at random. \newblock In {\em In and out of equilibrium ({M}ambucaba, 2000)}, volume~51 of  {\em Progr. Probab.}, pages 133--162. Birkh\"auser Boston, Boston, MA, 2002. \MR{1901951}

\bibitem{Maison} Bernard Maisonneuve. \newblock Exit systems. \newblock {\em Ann. Probability}, 3(3):399--411, 1975. \MR{0400417}

\bibitem{MY} Rogers Mansuy and Marc Yor. \newblock {\em Random times and enlargements of filtrations in a {B}rownian  setting}, volume 1873 of {\em Lecture Notes in Mathematics}. \newblock Springer-Verlag, Berlin, 2006. \MR{2200733}

\bibitem{MSW} P.~A. Meyer, R.~T. Smythe, and J.~B. Walsh. \newblock Birth and death of {M}arkov processes. \newblock In {\em Proceedings of the {S}ixth {B}erkeley {S}ymposium on  {M}athematical {S}tatistics and {P}robability ({U}niv. {C}alifornia,  {B}erkeley, {C}alif., 1970/1971), {V}ol. {III}: {P}robability theory}, pages  295--305. Univ. California Press, Berkeley, Calif., 1972. \MR{0405600}

\bibitem{Millarbis} P.~W. Millar. \newblock Zero-one laws and the minimum of a {M}arkov process. \newblock {\em Trans. Amer. Math. Soc.}, 226:365--391, 1977. \MR{0433606}

\bibitem{Millar} P.~W. Millar. \newblock A path decomposition for {M}arkov processes. \newblock {\em Ann. Probability}, 6(2):345--348, 1978. \MR{0461678}

\bibitem{Millarsurvey} P.~Warwick Millar. \newblock Random times and decomposition theorems. \newblock In {\em Probability ({P}roc. {S}ympos. {P}ure {M}ath., {V}ol. {XXXI},  {U}niv. {I}llinois, {U}rbana, {I}ll., 1976)}, pages 91--103. Amer. Math.  Soc., Providence, R. I., 1977. \MR{0443109}

\bibitem{Pitmantech}
J.~Pitman.
\newblock Path decomposition for conditional {B}rownian motion.
\newblock Technical Report~11, Inst. Math. Stat., Univ. of Copenhagen, 1974.

\bibitem{PitmanLevy} J.~W. Pitman. \newblock L\'evy systems and path decompositions. \newblock In {\em Seminar on {S}tochastic {P}rocesses, 1981 ({E}vanston,  {I}ll., 1981)}, volume~1 of {\em Progr. Prob. Statist.}, pages 79--110.  Birkh\"auser, Boston, Mass., 1981. \MR{0647782}

\bibitem{Pitmanexcursion} Jim Pitman. \newblock Stationary excursions. \newblock In {\em S\'eminaire de {P}robabilit\'es, {XXI}}, volume 1247 of {\em  Lecture Notes in Math.}, pages 289--302. Springer, Berlin, 1987. \MR{0941992}

\bibitem{PTpattern}
Jim Pitman and Wenpin Tang.
\newblock Patterns in random walks and {B}rownian motion.
\newblock 2014.
\newblock arXiv: 1411.0041.

\bibitem{PYl} Jim Pitman and Marc Yor. \newblock On the lengths of excursions of some {M}arkov processes. \newblock In {\em S\'eminaire de {P}robabilit\'es, {XXXI}}, volume 1655 of {\em  Lecture Notes in Math.}, pages 272--286. Springer, Berlin, 1997. \MR{1478737}

\bibitem{PSbis} A.~O. Pittenger and C.~T. Shih. \newblock Coterminal families and the strong {M}arkov property. \newblock {\em Bull. Amer. Math. Soc.}, 78:439--443, 1972. \MR{0297019}

\bibitem{PS} A.~O. Pittenger and C.~T. Shih. \newblock Coterminal families and the strong {M}arkov property. \newblock {\em Trans. Amer. Math. Soc.}, 182:1--42, 1973. \MR{0336827}

\bibitem{Revuz} D.~Revuz. \newblock Mesures associ\'ees aux fonctionnelles additives de {M}arkov. {I}. \newblock {\em Trans. Amer. Math. Soc.}, 148:501--531, 1970. \MR{0279890}

\bibitem{RY} Daniel Revuz and Marc Yor. \newblock {\em Continuous martingales and {B}rownian motion}, volume 293 of  {\em Grundlehren der Mathematischen Wissenschaften}. \newblock Springer-Verlag, Berlin, third edition, 1999. \MR{1725357}

\bibitem{Rost1} Hermann Rost. \newblock Markoff-{K}etten bei sich f\"ullenden {L}\"ochern im {Z}ustandsraum. \newblock {\em Ann. Inst. Fourier (Grenoble)}, 21(1):253--270, 1971. \MR{0299755}

\bibitem{Rost2} Hermann Rost. \newblock The stopping distributions of a {M}arkov {P}rocess. \newblock {\em Invent. Math.}, 14:1--16, 1971. \MR{0346920}

\bibitem{Sharpebook} Michael Sharpe. \newblock {\em General theory of {M}arkov processes}, volume 133 of {\em Pure  and Applied Mathematics}. \newblock Academic Press, Inc., Boston, MA, 1988. \MR{0958914}

\bibitem{SheppGaussian} L.~A. Shepp. \newblock Radon-{N}ikod\'ym derivatives of {G}aussian measures. \newblock {\em Ann. Math. Statist.}, 37:321--354, 1966. \MR{0190999}

\bibitem{Shepp} L.~A. Shepp. \newblock First passage time for a particular {G}aussian process. \newblock {\em Ann. Math. Statist.}, 42:946--951, 1971. \MR{0278375}

\bibitem{Slepian} D~Slepian. \newblock First passage time for a particular gaussian process. \newblock {\em The Annals of Mathematical Statistics}, 32(2):610--612, 1961. \MR{0125619}

\bibitem{Smirnov} N.~Smirnov. \newblock Table for estimating the goodness of fit of empirical distributions. \newblock {\em Ann. Math. Statistics}, 19:279--281, 1948. \MR{0025109}

\bibitem{HTP}
Hermann Thorisson.
\newblock Personal communications.

\bibitem{Th92} Hermann Thorisson. \newblock Construction of a stationary regenerative process. \newblock {\em Stochastic Process. Appl.}, 42(2):237--253, 1992. \MR{1176499}

\bibitem{Th1} Hermann Thorisson. \newblock Shift-coupling in continuous time. \newblock {\em Probab. Theory Related Fields}, 99(4):477--483, 1994. \MR{1288066}

\bibitem{Th1995} Hermann Thorisson. \newblock On time- and cycle-stationarity. \newblock {\em Stochastic Process. Appl.}, 55(2):183--209, 1995. \MR{1313019}

\bibitem{Th2} Hermann Thorisson. \newblock Transforming random elements and shifting random fields. \newblock {\em Ann. Probab.}, 24(4):2057--2064, 1996. \MR{1415240}

\bibitem{Th1999} Hermann Thorisson. \newblock Point-stationarity in {$d$} dimensions and {P}alm theory. \newblock {\em Bernoulli}, 5(5):797--831, 1999. \MR{1715440}

\bibitem{Thbook} Hermann Thorisson. \newblock {\em Coupling, stationarity, and regeneration}. \newblock Probability and its Applications (New York). Springer-Verlag, New  York, 2000. \MR{1741181}

\bibitem{vN} John Von~Neumann. \newblock Various techniques used in connection with random digits. \newblock {\em Applied Math Series}, 12(36-38):1, 1951. \MR{0045446}

\bibitem{Williams2} David Williams. \newblock Decomposing the {B}rownian path. \newblock {\em Bull. Amer. Math. Soc.}, 76:871--873, 1970. \MR{0258130}

\bibitem{Yor} Marc Yor. \newblock Grossissement d'une filtration et semi-martingales: th\'eor\`emes  g\'en\'eraux. \newblock In {\em S\'eminaire de {P}robabilit\'es, {XII} ({U}niv. {S}trasbourg,  {S}trasbourg, 1976/1977)}, volume 649 of {\em Lecture Notes in Math.}, pages  61--69. Springer, Berlin, 1978. \MR{0519996}

\end{thebibliography}



\ACKNO{We express our gratitude to Hermann Thorisson for helpful discussions, especially the constructive proof of Theorem \ref{mainbis}. We thank Alan Hammond for informing us the work \cite{HPS}, and Krzysztof Burdzy for pointing out the relevance of the recent work of Last and Thorisson \cite{LT2009}, and Last et al. \cite{LMT} to Question \ref{q1}. We also thank two anonymous referees for their careful reading and valuable suggestions.}


\end{document}